\documentclass[a4paper,10pt]{amsart}

\usepackage[all]{xy}

\usepackage{amsfonts}
\usepackage{amssymb}
\usepackage{soul}
\usepackage[table]{xcolor}
\usepackage[pdftex]{graphicx}
\usepackage{url} 

\newtheorem{theorem}{Theorem}   

\newtheorem{lemma}[theorem]{Lemma}
\newtheorem{proposition}[theorem]{Proposition}
\newtheorem{corollary}[theorem]{Corollary}

\theoremstyle{definition}
\newtheorem{definition}[theorem]{Definition}

\newtheorem{remark}[theorem]{Remark}

\definecolor{MyDarkBlue}{rgb}{0,0.08,0.60} 

\usepackage{tikz}
\usetikzlibrary{decorations.pathreplacing}

\newcommand{\A}{{ \rm Aut }}

\newcommand{\Z}{{\mathbb{Z}}}
\newcommand{\lc}{\left\lceil}
\newcommand{\rc}{\right\rceil}
\newcommand{\lf}{\left\lfloor}
\newcommand{\rf}{\right\rfloor}

\renewcommand{\mod}{{\;\rm mod}}

\DeclareMathOperator{\GL}{{GL}}
\DeclareMathOperator{\PGL}{{PGL}}
\DeclareMathOperator{\Der}{Der}
\DeclareMathOperator{\PDer}{PDer}
\DeclareMathOperator{\ord}{ord}

\newcommand\y{\cellcolor{red!40}}

\usepackage{amsmath,amssymb,amsfonts,amsthm}

\usepackage{todonotes}

\begin{document}

\title[Galois-module structure of Artin-Schreier-Mumford curves]{On the Galois-module structure of polydifferentials of Artin-Schreier-Mumford curves, modular and integral representation theory}

\date{\today}

\author{Aristides Kontogeorgis}
\address{Department of Mathematics, University of Athens Panepistimioupolis, 15784 Athens, Greece
}
\email{kontogar@math.uoa.gr}
\author{Dimitra-Dionysia Stergiopoulou}
\address{Department of Mathematics, University of Athens Panepistimioupolis, 15784 Athens, Greece
}
\email{dstergiop@math.uoa.gr}

\bibliographystyle{amsplain}

\begin{abstract}
We study the Galois-module structure of polydifferentials for Mumford curves, defined over a field of positive charactersitic, using the theory of harmonic cocycles. For the case of Artin-Schreier-Mumford curves the structure of holomorphic polydifferentials is explicitly computed.
\end{abstract} 


\maketitle

\section{Introduction}
Let $X$ be a smooth projective curve of genus $g\geq 2$ over an algebraically closed field $K$ of characteristic $p>0$, and $G$ a group of automorphisms of $X$.
The group $G$ acts on $X$ from the left, by our convention, and hence on the space of $n$-polydifferentials $H^0(X,\Omega_X^{\otimes n})$ from the right.
The so-called {\em Galois-module structure problem} for $X$ asks for the direct sum decomposition of $H^0(X,\Omega_X^{\otimes n})$ into $G$-indecomposable pieces.
In characteristic zero, the $n=1$ case is a classical result (\cite{Hu:1893}), which can be easily generalized for $n\geq 1$. 

In positive characteristic, the Galois-module structure is unknown in general.
There are only some partial results known. 
Let us give a brief overview. 
If the cover $X \rightarrow G\backslash X$ is unramified or if $(|G|,p)=1$, Tamagawa \cite{Tamagawa:51} determined the $G$-module structure of $H^0(X,\Omega_X)$. 
Valentini \cite{Val:82} generalized this result to unramified extensions with $G$ being a $p$-group. 
In the $p$-group case, moreover, Salvador and Bautista \cite{Salvador:00} determined the semi-simple part of the representation with respect to the Cartier operator. 
For the cyclic-group case, Valentini and Madan \cite{vm} and S.\ Karanikolopoulos \cite{SK} determined the structure of $H^0(X,\Omega_X)$ in terms of indecomposable modules. 
A similar study has been done for the elementary abelian case by {Calder{\'o}n, Salvador}  and Madan \cite{csm}.
Finally, N.\ Borne \cite{borne06} developed a theory, using advanced techniques from both modular representation theory and $K$-theory, for computing in some cases the  $K[G]$-{module} structure of the space of polydifferentials $H^0(X,\Omega_X^{\otimes n})$.


Let us point out that the determination of the Galois-module structure as above has several applications. 
For example, in \cite{KoJPAA06}, \cite{kockKo}, the second author connected the $K[G]$-module structure of $H^0(X,\Omega_X^{\otimes 2})$ to the computation of the tangent space of the global deformation functor of curves.

In this paper, we consider the Galois-module structure problem for the so-called {\em Artin-Schreier-Mumford curves} (see below). 
We give for these curves explicit bases of the space of polydifferentials, and apply the  theory of B.\ K\"ock \cite{Koeck:04} to complete spaces of polydifferentials by admitting controlled poles at certain points in order to obtain projective modules. 
By this way, we can prove that all indecomposable $K[G]$-modules admit $K[G]$ itself as an injective hull, and finally arrive at our main results.


Over a complete discrete valuation field $K$, D.\ Mumford \cite{mumford-curves} has shown that a smooth projective curve with the split multiplicative reduction is isomorphic to the algebraization of a rigid analytic curve over $K$ of the form $\Gamma\backslash(\mathbb{P}^{1,\mathrm{an}}_K - \mathcal{L}_\Gamma)$. 
Here, $\Gamma$ is a finitely generated torsion-free discrete subgroup of $\mathrm{PGL}(2,K)$, called a {\em Schottky group}, and $\mathcal{L}_\Gamma$ is the set of limit points.
A smooth projective curve obtained in this way, denoted by $X_{\Gamma}$, is called a \emph{Mumford curve}, and the uniformization just described provides us with a set of tools similar to those coming from the uniformization theory of Riemann surfaces.
It is known that the subgroup $\Gamma$ is always a free group of finite rank, and the rank is equal to the genus of $X_{\Gamma}$.
The 
authors together with G.\ Cornelissen have used this technique in order to bound the automorphism groups of Mumford curves in \cite{CKK}. 
In fact, the automorphism group $\A(X_\Gamma)$ is isomorphic to the quotient $N_{\Gamma}/\Gamma$ of the normalizer of $\Gamma$ in $\mathrm{PGL}(2,K)$ by $\Gamma$; cf.\ \cite[1.3]{CKK} and \cite[VII.1)]{Ge-Pu}. 
Also the equivariant deformation theory of such curves was studied by the first author and G.\ Cornelissen in \cite{CK}.

One of the tools we will use is the explicit description of polydifferentials in terms of harmonic cocycles. 
P.\ Schneider and J.\ Teitelbaum \cite{Schneider84}\cite{Teitelbaum90}, defined the space of modular forms (or harmonic measures as they are known in the literature) $C_{\mathrm{har}}(\Gamma,n)$ on the reduction graph, and they showed  that it is naturally isomorphic to $H^0(X_\Gamma,\Omega_{X_\Gamma}^{\otimes n})$.
Moreover, the space $C_{\mathrm{har}}(\Gamma,n)$ can be described by the Galois cohomology $C_{\mathrm{har}}(\Gamma,n)\cong H^1(\Gamma, P_n)$, where $P_n$ denotes the space of polynomials of one variable of degree $\leq 2n-2$ (cf.\ \S\ref{sec2} for more details).

Now let us state our main results of this paper. 
We first give the definition of Artin-Schreier-Mumford curves:
\begin{definition}\label{def-ASM}
Let $K$ be a complete non-archimedean valued field of characteristic $p>0$, and $q$ a power of $p$.
For $\lambda\in K$ with $0<|\lambda|<1$, the smooth projective model of the affine plane curve defined by the equation 
\[
(x^q-x)(y^q-y)=\lambda
\]
will be called an {\em Artin-Schreier-Mumford curve}.
\end{definition}
These curves are special from quite a few points of view. 
For example, they are the Mumford curves with maximal automorphism group (and hence their Schottky groups are the analogue of classical Hurwitz groups), cf.\ \cite{CKK} and \cite{CK-Proc}. 
They were first studied by D.\ Subrao \cite{Subrao}, Valentini-Madan \cite{vm}, and S.\ Nakajima \cite{Nak:86}.
M.\ Matignon has studied their equivariant liftability to characteristic zero \cite{matlift}, and these curves play a special role when studying the `field of definition versus field of moduli' question for cyclic covers of the projective line (cf.\ \cite{Kont-Bord}).

In this paper, we only deal with Artin-Schreier-Mumford curves as in Definition \ref{def-ASM} with $q=p$. 
For the proof of the following facts, we refer to \cite[\S9]{CKK} and \cite[p.\ 347]{CKK-israel}:
\begin{proposition} \label{prop-Schottky}
The Artin-Schreier-Mumford curves are Mumford curves of the form $X_\Gamma$, where the group $\Gamma$ is, up to conjugacy, given by the commutator group $\Gamma=[A,B]$ of the cyclic subgroups $A,B \subset \mathrm{PGL}(2,K)$ of order $p$ generated by 
\begin{equation} \label{matrices}
\epsilon_A=\begin{pmatrix} 1 & 1 \\ 0 & 1 \end{pmatrix}\quad\textrm{and}\quad\epsilon_B=\begin{pmatrix} 1 & 0 \\ s & 1 \end{pmatrix}, 
\end{equation}
respectively, where $s\in K^{\times}$ and $|s|>1$.
The groups $A$ and $B$ generate a discrete subgroup $N\subseteq\PGL(2,K)$, which is isomorphic to the free product $A\ast B$.  
Moreover$:$
\begin{itemize}
\item[{\rm (a)}] $\Gamma$ is a normal subgroup of $N$ and $N/\Gamma\cong A\times B;$
\item[{\rm (b)}] $\Gamma$ is a free group of rank $(p-1)^2$ with the basis given by the commutators $e_{i,j}=[\epsilon^i_A,\epsilon^j_B]$ $(=\epsilon^i_A\epsilon^j_B\epsilon^{-i}_A\epsilon^{-j}_B)$ for $i,j=1,\ldots,p-1$. \hfill$\square$
\end{itemize}
\end{proposition}


\begin{remark}
The relation between the parameter $\lambda$ in Definition \ref{def-ASM} and the parameter $s$ in Proposition \ref{prop-Schottky} has been studied in \cite{CKK-israel}.
\end{remark}

It has been shown in \cite[\S9]{CKK} that the automorphism group of the Artin-Schreier-Mumford curve $X_{\Gamma}$ contains $G=\Z/p\Z \times \Z/p\Z$, generated by the images of $\epsilon_A$ and $\epsilon_B$ in $\A(X_\Gamma)\cong N_{\Gamma}/\Gamma$.

The first result of this paper gives the $K[G]$-module structure of the space of $1$-differentials:
\begin{theorem}\label{theoremnis1}
{\rm (1)} As a $K[A]$-module, we have
\[
H^0(X_\Gamma,\Omega_X)\cong L^{p-1}\otimes_\Z K,
\]
where $L$ is the integral representation of $A\cong\Z/p\Z$ with the minimal rank $p-1$ $($corresponding to the matrix $M$ in {\rm (\ref{protype})} below$)$.

{\rm (2)} As a $K[G]$-module, $H^0(X_\Gamma,\Omega_X)$ is indecomposable. 
\end{theorem}

Notice that, since the space of $1$-differentials can be expressed combinatorially, the $K[A]$-module structure actually comes from an integral representation as in Theorem \ref{theoremnis1} (1), which is, however, not the case for higher polydifferentials.

\begin{theorem} \label{mainFree}
Suppose $p\neq 2$. 
For $n>1$, let $r$ $(0\leq r<p)$ be the remainder of $2n-1$ modulo $p$.

{\rm (1)} As a $K[A]$-module, the following decomposition holds$:$
$$
H^0(X_\Gamma,\Omega_{X_\Gamma}^{\otimes n})\cong K[A]^{ (p-1) (2n-1) - p \lc \frac{2n-1}{p} \rc } \oplus \left( K[A]/(\epsilon_A-1)^{p-r}\right)^p. 
$$
A similar result holds for the group $B$. 

{\rm (2)} As a $K[G]$-module, the following decomposition holds$:$  
\[
 H^0(X_{\Gamma},\Omega_{X_\Gamma}^{\otimes n})\cong K[G]^{2n-1 - 2\lc \frac{2n-1}{p} \rc} \oplus  K[G]/(\epsilon_A-1)^{p-r}
\oplus K[G]/(\epsilon_B-1)^{p-r}.
\]
\end{theorem}

Let us now describe the structure of this paper. 
The next section (\S\ref{sec2}) recalls the description of the space of polydifferentials of Mumford curves in terms of the group cohomologies.
In \S\ref{sec3} we focus on $1$-differentials. 
As a side result, we obtain a bound for the order of an automorphism acting on them (see Corollary \ref{cor-orderupperbound}).
We also give in this section a criterion for a module to be indecomposable, based on the dimension of the space of invariant elements. 
From \S\ref{favorCont} onward, we proceed to the study of the space of polydifferentials. 
In \S\ref{6} we first study  $K[A]$-module structure using a combinatorial approach. 
Then we also show how results of S.\ Nakajima \cite{Nak:86} can be applied without the usage of the theory of Mumford curves.
For the $K[G]$-module structure, we employ both the theory of projective covers and the theory of B.\ K\"ock on the Galois-module structure of weakly ramified covers. 

\medskip
{\bf Conventions.}
For a ring $R$ and a group $G$, we denote by $R[G]$ the group ring over $R$. 
As for $R[G]$-modules, we always consider {\em right} $R[G]$-modules, unless otherwise clearly stated.
If an $R[G]$-module $V$ is finite free as an $R$-module, then, for any $\gamma\in G$, the {\em matrix representation} of $\gamma$ by an $R$-basis $\{v_1,\ldots,v_r\}$ of $V$ is the matrix $M_{\gamma}\in\GL(r,R)$ whose $i$-th {\em row} is given by $(a_1,\ldots,a_r)$, where $v^{\gamma}_i=\sum_{j=1}^ra_jv_j$. Notice that, by this way, the map $G\rightarrow\GL(r,R)$ by $\gamma\mapsto M_{\gamma}$ is a group homomorphism.
Accordingly, Jordan matrices in our sense will be the transpose of the conventional ones, having $1$'s on the {\em lower} subdiagonal entries.


\medskip
\noindent {\bf Acknowledgements}
We would like to thank Janne Kool and Fumiharu Kato for their input at an early stage of this project. 
 The research project is implemented in the framework of H.F.R.I Call “Basic research Financing Horizontal support of all Sciences)” under the National Recovery and Resilience Plan “Greece 2.0” funded by the European Union Next Generation EU, H.F.R.I.  
Project Number: 14907.
\begin{center}
\includegraphics[scale=0.4]{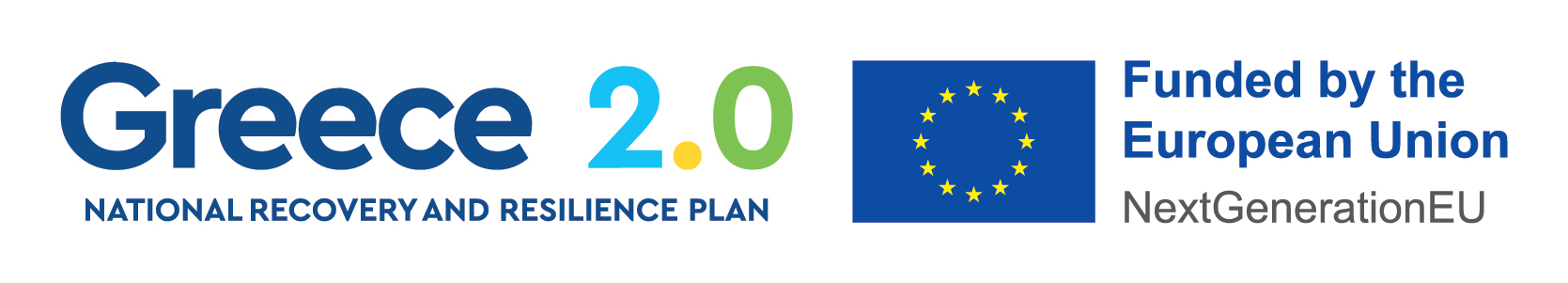}
\hskip 1cm
\includegraphics[scale=0.05]{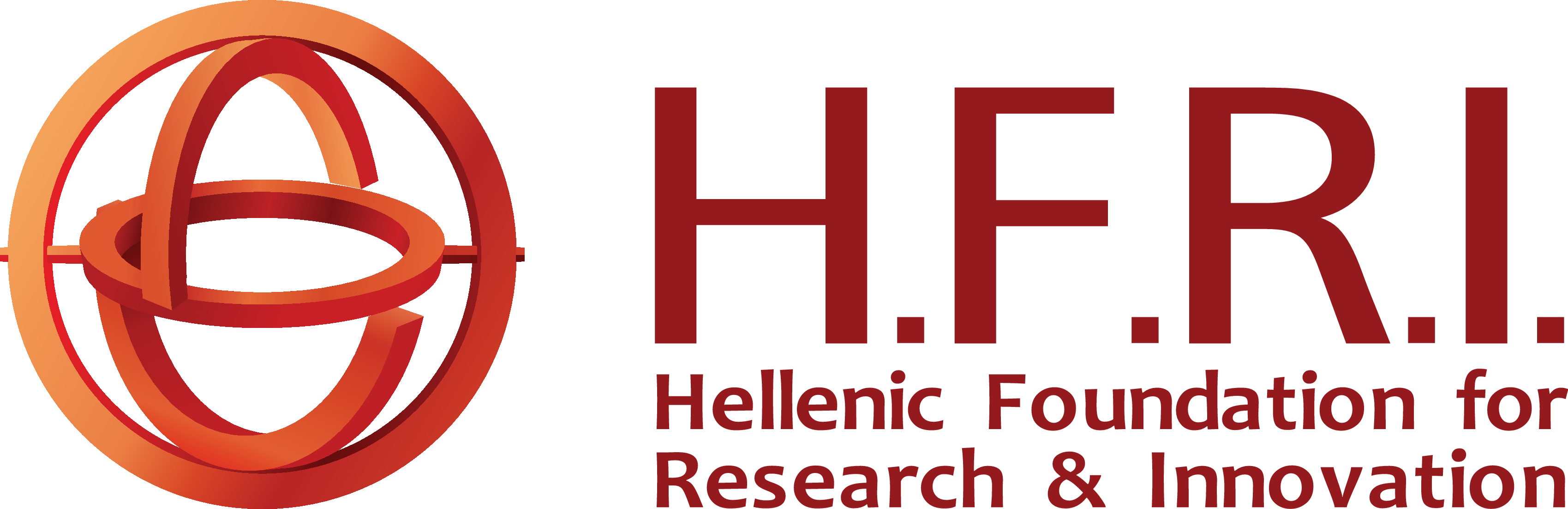}
\end{center}


\section{Preliminaries} \label{sec2}
\subsection{Invariants and direct factors}\label{3.1}
Let $K$ be an algebraically closed field of characteristic $p>0$, $G\cong\Z/p\Z$ a cyclic group of order $p$, and $\sigma\in G$ a generator.
For each $1\leq\mu\leq p$, consider the $\mu\times\mu$ Jordan matrix $J_{\mu}\in\GL(\mu,K)$ (with the $1$'s on the lower subdiagonal entries) of eigenvalue $1$.
Then the $\mu$-dimensional $K$-vector space $K^{\mu}$ can be regarded as a right $K[G]$-module by $\sigma\mapsto J_{\mu}$.
By a slight abuse of notation, we denote thus obtained $K[G]$-module by the same notation $J_{\mu}$; notice that $J_{\mu}$ is an indecomposable $K[G]$-module, isomorphic to $K[G]/((\sigma-1)^{\mu})$ (cf.\ \cite[\S1 p.\ 107]{vm}).

If $V$ is an arbitrary finite dimensional $K[G]$-module, then, by taking Jordan normal forms, one has the indecomposable decomposition of the form $V\cong\bigoplus^r_{i=1}J_{\mu_i}$ as $K[G]$-module, where $r\geq 0$ and $1\leq\mu_i\leq p$ for each $i=1,\ldots,r$.

\begin{proposition}\label{numberofsummands}
Let $G$ be a finite cyclic $p$-group, and $V$ a finite dimensional $K[G]$-module. 
Then the number of indecomposable $K[G]$-summands of $V$ is equal to $\dim_KV^G$.
\end{proposition}

\begin{proof}
The indecomposable $K[G]$-summands of $V$ are in one-to-one correspondence with the blocks of the Jordan normal form of a generator $\sigma$ of $G$, seen as an element of $\mathrm{GL}(V)$.
Since a Jordan block has the one-dimensional invariant subspace, every direct summand contributes exactly an one dimensional invariant subspace.  
\end{proof}

\begin{remark}
The assumption that $G$ is cyclic is necessary. 
See, for example, the $K[\Z/p\Z\times \Z/p\Z]$-module given by Heller and Reiner in \cite[Example 1.4, p.\ 157]{Wein}. 
\end{remark}

\begin{corollary}\label{cor-11}
Let $G$ be an abelian $p$-group acting on a finite dimensional non-zero $K$-vector space $V$.
Then we have $V^G\neq\{0\}$. 
\end{corollary}

\begin{proof}
The group $G$ is isomorphic to a direct product of cyclic $p$-groups. 
The proof follows by induction with respect to the number of the direct factors, aided with the fact $M^{H_1\times H_2}=(M^{H_1})^{H_2}$ and Proposition \ref{numberofsummands}.
%
\end{proof}
%


\begin{corollary} \label{14}
In the situation as in Corollary {\rm \ref{cor-11}}, if $\dim_KV^G=1$, then $V$ is indecomposable. \hfill$\square$
\end{corollary}


\subsection{Derivations and the group cohomology}\label{3.1bis}
Let $K$ be a field, $G$ a group, and $P$ a right $K[G]$-module.
A {\em derivation} (or {\em $1$-cocycle}) of $G$ to $P$ is a map $d\colon\Gamma\rightarrow P$ satisfying 
\begin{equation}\label{eqn-onecocycle}
 d(\gamma \gamma')= (d\gamma)^{\gamma'} + d\gamma'
\end{equation}
for any $\gamma,\gamma' \in G$. 
In particular, for $\gamma\in G$ and a integer $k\geq 0$, we have
\begin{equation}\label{eqn-onecocycle2}
 d(\gamma^k)= (d\gamma)^{1+\gamma+\cdots+\gamma^{k-1}}.
\end{equation}
The set of all derivations $\mathrm{Der}(G,P)$ is naturally a $K$-linear space.
A {\em principal derivation} (or {\em $1$-coboundary}) is a derivation of the 
form 
\[
 G \ni \gamma \mapsto F^\gamma -F, 
\]
by an element $F\in P$. 
Principal derivations form a subspace $\PDer(G,P)$ of $\Der(G,P)$. 
The quotient is the ($1$st) \emph{group cohomology}:
\[
 H^1(G,P)=\Der(G,P)/\PDer(G,P).
\]

\subsection{Polydifferentials on Mumford curves}\label{sub-polymumford}
Now, let $K$ be a complete non-archimedean valued field, $\Gamma\subseteq\PGL(2,K)$ a Schottky subgroup, and $X_{\Gamma}$ the Mumford curve obtained from $\Gamma$.
If $N=N_{\Gamma}\subseteq\PGL(2,K)$ is the normalization of $\Gamma$ in $\PGL(2,K)$, the quotient group $G=N/\Gamma$, which acts on $X_{\Gamma}$ from the left, is the automorphism group $\A(X_\Gamma)$ of $X_{\Gamma}$ over $K$.
Notice that the group $\Gamma$ is a free group of finite rank, whose rank, say $g$, is equal to the genus of $X_{\Gamma}$.
We fix a free generating set $\{\gamma_1,\ldots,\gamma_g\}$ of $\Gamma$.

For any right $K[\Gamma]$-module $P$, each derivation $d\colon\Gamma\rightarrow P$ is uniquely determined by its values $h_i=d(\gamma_i)$ for $1\leq i\leq g$, and conversely, since $\Gamma$ is free, such values $h_i\in P$ can be freely chosen to obtain a derivation $d$; indeed, once $h_i$'s are chosen, then $d(w)$ for any $w\in\Gamma$ is uniquely determined by the recursive application of (\ref{eqn-onecocycle}).

For a positive integer $n$, we consider the space of polynomials $P_n\subseteq K[T]$ of degree $\leq 2(n-1)$, which is a $K$-vector space of dimension $2n-1$. 
The group $\mathrm{PGL}(2,K)$ acts on $P_n$ from the right as follows: for 
$\gamma=\begin{pmatrix} a & b \\ c & d \end{pmatrix} \in \mathrm{PGL}(2,K)$ and $F \in P_n$, we define 
\begin{equation} \label{pol-action}
 F^{\gamma}(T):=\frac{(cT+d)^{2(n-1)} }{ (ad-bc)^{n-1} } F\left( \frac{aT+b}{cT+d} \right) \in K[T].
\end{equation}

Now, consider the space $\mathrm{Der}(\Gamma,P_n)$ of derivations.
By what we have seen above, this is a $K$-linear space of dimension $(2n-1)g$.
The space $\mathrm{Der}(\Gamma,P_n)$ admits a right action of $N$ (and hence of the group algebra $K[N]$) as follows: for $\delta\in N$ and $d \in \mathrm{Der}(\Gamma,P_n)$, define
\begin{equation} \label{der-act}
 (d^{\delta})(\gamma)=[d(\delta\gamma\delta^{-1})]^{\delta}
\end{equation}
for $\gamma\in\Gamma$.
We have thus a well-defined right action of $G=N/\Gamma$ on the group cohomology $H^1(\Gamma,P_n)$, since $\Gamma$ acts trivially modulo principal derivations: 
\[
[d(\delta\gamma\delta^{-1})]^{\delta}=d(\delta)^\gamma - d(\delta) + d(\gamma) \mbox{ for } \delta, \gamma \in  \Gamma. 
\]


\begin{theorem}[{\rm \cite[Theorem 1]{Teitelbaum90}}] \label{th5}
For any $n\geq 1$, the space $H^0(X_{\Gamma},\Omega^{\otimes n}_{X_{\Gamma}})$ of $n$-differentials on the curve $X_\Gamma$ is naturally isomorphic to the space group cohomology $H^1(\Gamma,P_n)$.
Moreover, this identification is $G$-equivariant with respect to the natural right $G$-action on $H^0(X_{\Gamma},\Omega^{\otimes n}_{X_{\Gamma}})$. \hfill$\square$
\end{theorem}

%
%
%
\section{The space of $1$-differentials} \label{sec3}
\subsection{{\boldmath $G$-action on $1$-differentials}}
We continue to work with the notation as in \S\ref{sub-polymumford}, and suppose $K$ is of characteristic $p>0$.
By Theorem \ref{th5}, we have
\begin{equation}
\begin{split}
H^{0}(X_\Gamma,\Omega_{X_\Gamma}) &\cong H^1(\Gamma,P_0)=H^1(\Gamma,K)=
 \mathrm{Hom}(\Gamma,K)=\mathrm{Hom}(\Gamma,\Z)\otimes K  \\ 
 &\cong\mathrm{Hom}(\Gamma^{\mathrm{ab}},\Z) \otimes K, \label{hol-iso}
\end{split}
\end{equation}
where $\Gamma^{\mathrm{ab}}\cong\Z^g$ denotes the maximal abelian quotient of the free group $\Gamma$. 
Since $G=N/\Gamma$ acts on $\Gamma^{\mathrm{ab}}$ from the left (by conjugation), we have the right action of $G$ on $\mathrm{Hom}(\Gamma^{\mathrm{ab}},\Z)$, and hence on $\mathrm{Hom}(\Gamma^{\mathrm{ab}},\Z)\otimes K$.
The isomorphism $H^{0}(X_\Gamma,\Omega_{X_\Gamma})\cong\mathrm{Hom}(\Gamma^{\mathrm{ab}},\Z) \otimes K$ by (\ref{hol-iso}) is easily seen to be $G$-equivariant.
In particular, the $G$-action on $H^{0}(X_\Gamma,\Omega_{X_\Gamma})$ comes from an integral representation $\rho\colon G\rightarrow \mathrm{GL}(g,\Z)$.
Then, by \cite{KockF}, we have:
\begin{proposition}\label{faithful}
The integral representation $\rho$ is injective, i.e., $G$ can be seen as a subgroup of $\mathrm{GL}(g,\mathbb{Z})$, unless the cover $X\rightarrow G\backslash X=Y$ is not tamely ramified, the characteristic of $K$ is $2$, and the genus of $Y$ is $0$. \hfill$\square$
\end{proposition}




\begin{corollary}\label{cor-orderupperbound}
Suppose $p\neq 2$. 
If the order of an element $g'\in G$ is a prime number $q$, then $q \leq g+1$.
\end{corollary}

\begin{proof}
This follows from Proposition \ref{faithful} and a special case of \cite[Theorem 2.7]{KuzPav}.
\end{proof}

%

\subsection{Proof of Theorem \ref{theoremnis1}} \label{freeG}
Let $A=\langle \epsilon_A \rangle$, $B=\langle \epsilon_B \rangle$, $\Gamma$, $N$, and $G=N/\Gamma$ be as in Proposition \ref{prop-Schottky}. 
The set $\{e_{i,j}=[\epsilon_A^i,\epsilon_B^j]\,|\,1\leq i,j \leq p-1\}$ gives a basis of $\Gamma$ (cf.\ \cite[p.\ 6, Prop.\ 4]{SerreTrees}).
%
By an easy calculation, we have
\begin{equation}
\label{eq:eAactions}
\epsilon_A e_{i,j}\epsilon^{-1}_A=[\epsilon_A^{i+1},\epsilon_B^j][\epsilon_A,\epsilon_B^j]^{-1}\quad\textrm{for}\quad 1\leq i,j\leq p-1,
\end{equation}
which describes the left action of $A$ on $\Gamma$.

For any $\gamma\in\Gamma$, let us denote by $\overline{\gamma}$ the image of $\gamma$ in the maximal abelian quotient $\Gamma^{\mathrm{ab}}\cong \Z^g$.
The free abelian group $\Gamma^{\mathrm{ab}}$ has the $\Z$-basis consisting of $\overline{e}_{i,j}$'s.
Let $\{f_{i,j}\}$ be the dual basis in $V=\mathrm{Hom}(\Gamma^{\mathrm{ab}},\Z)\otimes K\cong H^0(X_{\Gamma},\Omega_{X_{\Gamma}})$ of $\{\overline{e}_{i,j}\}$.
Since, written additively in $\Gamma^{\mathrm{ab}}\cong \Z^g$, we have
$$
\overline{\epsilon_A e_{i,j}\epsilon^{-1}_A}=\begin{cases} \overline{e}_{i+1,j}-\overline{e}_{1,j} & \textrm{for $1\leq i \leq p-2$,}\\ -\overline{e}_{1,j} &\textrm{if $i=p-1$,}\end{cases}
$$
for $1\leq j \leq p-1$, we have
\begin{equation}\label{eqn-dualbasis}
f^{\epsilon_A}_{i,j}=\begin{cases}{\displaystyle -\sum^{p-1}_{k=1}f_{k,j}}&\textrm{if $i=1$,}\\ f_{i-1,j}&\textrm{for $2\leq i\leq p-1$.}
\end{cases}
\end{equation}
\begin{remark}
Usually on the dual space we act in terms of the contragredient representation in order to have a left action on dual elements as well. Here the action on $f:V\rightarrow k$ is given by $f \mapsto f^g$, where $f^g$ is the function sending $v\mapsto f(gv)$. 
\end{remark}
Hence he matrix representation of $\epsilon_A$ with respect to the basis $\{f_{i,j}\}$ is given by the block diagonal matrix $\mathrm{diag}(M,\ldots,M)$ consisting of $p-1$ copies (indexed by $j=1,\ldots, p-1$) of
\begin{equation} \label{protype}
 M= 
 \begin{pmatrix}
  -1 & -1 & -1 & \cdots & -1 \\
   1 &  0 & \cdots & \cdots & 0         \\
   0 &  1 &  0       &  & \vdots      \\
   \vdots & \ddots& \ddots & \ddots & \vdots \\
   0      & \cdots & 0  & 1 & 0\\
 \end{pmatrix} \in \mathrm{GL}(p-1,\Z)
\end{equation}
(cf.\ our convention for the matrix representation in Introduction).
Notice that the matrix $M$ has characteristic polynomial $\frac{x^p-1}{x-1}=1+x+\cdots +x^{p-1}$, of which $M$ is the companion matrix.
The integral representation of $\Z/p\Z$ on $\Z^{p-1}$ given by the matrix $M$, denoted by $L$ in the sequel, is the one with the minimal degree $p-1$; cf.\ \cite{KuzPav}. 
We thus have proven the first part of Theorem \ref{theoremnis1}.


To proceed, let us compute the invariant part by the action of $A$. 
It suffices to look at each block for $1\leq j \leq p-1$.
The condition for an element $\sum_{i=1}^{p-1} \lambda_i f_{i,j}$ to be invariant is given by
\[
\sum_{i=1}^{p-1} \lambda_i f_{i,j}=
 \left( \sum_{i= 1}^{p-1}  \lambda_i f_{i,j}\right)^{\epsilon_A}=
-\lambda_1\sum_{i= 1}^{p-1} f_{i,j}+\sum_{i= 1}^{p-2} \lambda_{i+1} f_{i,j},
\]
which is equivalent to $\lambda_i=i\cdot\lambda_1$ for $i=1,\ldots,p-1$.
Hence, each diagonal block contributes $1$ to the dimension of the space of invariants.
Since there are $p-1$ of them, the space of $A$-invariants in $V=\mathrm{Hom}(\Gamma^{\mathrm{ab}},\Z)\otimes K\cong H^0(X_{\Gamma},\Omega_{X_{\Gamma}})$ has dimension $p-1$, generated by the elements
\[
 f_j=\sum_{i= 1}^{p-1} i\cdot f_{i,j}  \mbox{ for } 1\leq j \leq p-1.
\]

We can now compute the space of $A\times B$-invariants by using the fact $V^{A\times B}=(V^A)^B$. 
Similarly to (\ref{eqn-dualbasis}), one computes
\begin{equation}\label{eqn-dualbasis2}
f^{\epsilon_B}_{i,j}=\begin{cases}{\displaystyle -\sum^{p-1}_{k=1}f_{i,k}}&\textrm{if $j=1$,}\\ f_{i,j-1}&\textrm{for $2\leq j\leq p-1$.}
\end{cases}
\end{equation}
From this, one has 
\begin{equation}\label{eqn-dualbasis3}
f^{\epsilon_B}_j=\begin{cases}{\displaystyle -\sum^{p-1}_{k=1}f_k}&\textrm{if $j=1$,}\\ f_{j-1}&\textrm{for $2\leq j\leq p-1$,}
\end{cases}
\end{equation}
which means that the matrix representation of $\epsilon_B$ with respect to the basis $\{f_j\}$ of the space of $A$-invariants coincides with the one in (\ref{protype}).
Hence, the space $V^{A\times B}$ is one dimensional and the representation is indecomposable by Corollary \ref{14}, which finishes the proof of the second part of Theorem \ref{theoremnis1}.

\section{Computations on Artin-Schreier-Mumford  curves continued.} \label{favorCont}
We continue with the notation of the previous subsection.
In this section, as a preparation for the proof of Theorem \ref{mainFree}, we first compute the space $H^1(\Gamma,P_n)^G$, and study the $K[A]$-module structure of $\mathrm{Der}(\Gamma,P_n)$.
(See \S\ref{sec2} for the definition of $P_n$.)


\subsection{Computation of the space {\boldmath $H^1(\Gamma,P_n)^G$}}\label{sec-5.1}
We first of all prove:
\begin{proposition} \label{Gamma-invar} 
We have $P^{\Gamma}_n=H^0(\Gamma,P_n)=\{0\}$ for $n>1$.
\end{proposition}

To show the proposition, we need the following lemma: 
\begin{lemma}\label{lem-entirefixedpoint}
For any discrete free subgroup $\Gamma\subseteq\PGL(2,K)$ of finite rank $\geq 2$ and any closed point $x$ of $\mathbb{P}^1_K$, the $\Gamma$-orbit $\Gamma\cdot x=\{\gamma x\,|\,\gamma\in\Gamma\}$ is an infinite set.
\end{lemma}

\begin{proof}
Suppose $\Gamma\cdot x$ is finite.
If $\{\gamma_1,\ldots,\gamma_g\}$ $(g\geq 2)$ is a free basis of $\Gamma$, then there exists an integer $N\geq 1$ such that $\gamma^N_ix=x$ for any $1\leq i\leq g$.
Since the subgroup generated by $\gamma^N_1,\ldots,\gamma^N_g$ is discrete and free of rank $g$, we may replace $\Gamma$ by this subgroup, and thus may assume that $x$ is fixed by every element in $\Gamma$.
Since $g\geq 2$, one can find two $\gamma,\delta\in\Gamma$ that share exactly one point $x$ as their fixed points.
Then one sees easily (cf.\ the proof of \cite[4.2 (3)]{OrbifoldKato}) that $[\gamma,\delta]\in\Gamma$ is a parabolic element, and hence is of order $p$, which is absurd.
\end{proof}

\begin{proof}[{\it Proof of Proposition {\rm \ref{Gamma-invar}}}]
Let $F\in P^{\Gamma}_n$.
Notice first that, since $n>1$, $F$ cannot be a non-zero constant; indeed, there exists an element $\begin{pmatrix}a&b\\ c&d \end{pmatrix}\in\Gamma$ with $c\neq 0$ (e.g., $[\epsilon_A,\epsilon_B]$).
Hence, if $F\neq 0$, there exists an irreducible polynomial over $K$ that divides $F$.
On the other hand, it can be checked by an easy calculation that, if $\rho\in\overline{K}$ is a root of $F$, then for any $\gamma=\begin{pmatrix}a&b\\ c&d \end{pmatrix}\in\PGL(2,K)$, ${\displaystyle \gamma^{-1}(\rho)=\frac{d\rho-b}{-c\rho+a}}$ is a root of $F^{\gamma}$.
Hence, by Lemma \ref{lem-entirefixedpoint}, $F$ has to be divided by infinity many irreducible polynomials, which is absurd.
\end{proof}

\begin{corollary} \label{invCoh} 
For $n>1$, we have  
\[
 H^1(N,P_n)\cong H^1(\Gamma,P_n)^G.
\]
\end{corollary}

\begin{proof}
 Consider the $5$-term restriction-inflation sequence  coming from the Lyndon-Hochschild-Serre spectral sequence \cite[par. 6.8.3]{Weibel}: 
$$
0 \rightarrow H^1(G,P^{\Gamma}_n)\stackrel{\mathrm{inf}}{\longrightarrow}
H^1(N,P_n) \stackrel{\mathrm{res}}{\longrightarrow} H^1(\Gamma,P_n)^G\rightarrow H^2(G,P^{\Gamma}_n)\rightarrow  H^2(N,P_n). 
$$
By Proposition \ref{Gamma-invar}, we have $H^1(G,P^{\Gamma}_n)= H^2(G,P^{\Gamma}_n)=\{0\}$, whence the result.
\end{proof}

\begin{remark}
For $n=1$, we have $P_0=K$, and we compute 
\begin{equation}\label{K2}
H^1(G,P^{\Gamma}_0) \cong H^1(N,P_0)\cong K^2.
\end{equation}
Indeed, the action of $N\cong A\ast B$ on $K$ is trivial, and hence by \cite[Ex.\ 6.2.5, p.\ 171]{Weibel}, we have 
$$
H^1(N,P_0)\cong H^1(A,K) \times H^1(B,K)\cong K^2.
$$
On the other hand, by \cite[\S3.5, p.\ 32]{Evens}, the cohomology ring $H^{\ast}(G,K)$ (recall that $G\cong A\times B$) is of the form
$$
H^*(G,K)\cong\bigwedge [\eta_1,\eta_2] \otimes k[\xi_1,\xi_2],
$$
where $\deg\eta_i=1$, $\deg\xi_i=2$, $\eta_i^2=0$.  
The degree-$1$ part is the two dimensional vector space spanned by $\eta_1,\eta_2$, hence we deduce that the inflation map $H^1(G,P^{\Gamma}_0)\hookrightarrow H^1(N,P_0)$ is an isomorphism, obtaining the isomorphisms as in (\ref{K2}).
Notice that $H^2(G,K)$ is of dimension three, generated by $\eta_1 \wedge \eta_2, \xi_1,\xi_2$, while the space $H^2(N,K)\cong H^2(A,K) \times H^2(B,K)$ (by \cite[Cor.\ 6.2.10, p.\ 170]{Weibel}) is two-dimensional, being compatible with the computation of invariants done in section \ref{freeG}, and the map $H^2(G,K) \rightarrow H^2(N,K)$ is surjective.
\end{remark}


To proceed, we need a convenient basis for the space $P_n$. 
Consider
\begin{equation} \label{binomial-def}
\binom{T}{k}=\frac{T(T-1)(T-2)\cdots (T-k+1)}{k!} \in K[T],
\end{equation}
for $k\geq 0$, which is a polynomial of degree $k$. 
For $k<0$, we set $\binom{T}{k}=0$. 
For each $k=0,\ldots,2(n-1)$, let $q$ and $r$ be the integers such that $k=qp+r$ and $0\leq r<p$, and define
\begin{equation} \label{bk}
b_k= \left( T^p-T \right)^q\binom{T}{r}.
\end{equation}
The elements $b_k$ for $0 \leq k \leq 2(n-1)$ form a basis of $P_n$, and using the binomial relation, we have
\begin{equation}\label{eqn-binomialrel}
b^{\epsilon_A}_k=\begin{cases} b_k+b_{k-1} &\textrm{if $p\nmid k$,}\\ b_k&\textrm{if $p|k$.}\end{cases}
\end{equation}
Hence $\epsilon_A$ in terms of $\{b_k\}$ is expressed by the block diagonal matrix 
\begin{equation}
\label{eqn-EA}
E_A=\mathrm{diag}(J_p,\ldots,J_p,J_r)
\end{equation}
consisting of $\lf\frac{2n-1}{p} \rf$ copies of $J_p$ and $J_r$ (cf.\ \S\ref{3.1} for the notation), where $0\leq r<p$ is the remainder of $2n-1$ modulo $p$; here, we put $J_0=\{0\}$ for convenience.

\begin{proposition}\label{prop-dimensions1}
Suppose $n>1$. 

{\rm (1)} We have $\dim_K\Der(A,P_n)=\dim_K\Der(B,P_n)=2n-1-\lf \frac{2n-1}{p} \rf$.

{\rm (2)} We have $\dim_KH^1(N,P_n)=2n-1 -2 \lf \frac{2n-1}{p}\rf$.
\end{proposition}

To show the proposition we the following lemma:
\begin{lemma} \label{sums-powerJord}
Consider the Jordan matrix $J_{\mu}$ for $1\leq\mu\leq p$.
We have 
\[
I_{\mu}+J_\mu+J_\mu^2+ \cdots J_\mu^{p-1}=
\begin{cases}
0 & \mbox{ if } \mu<p, \\
\begin{pmatrix}
0 & 0 & \cdots & 0 \\
\vdots & \vdots & \ddots  & \vdots \\
0 & 0 & \cdots & 0 \\
1 & 0 & \cdots & 0 \\
\end{pmatrix} & \mbox{ if } \mu=p.
\end{cases}
\]
\end{lemma}

\begin{proof}
Set $N_{\mu}=J_{\mu}-I_{\mu}$, which is the nilpotent Jordan matrix.
Then the assertion follows immediately from the direct calculation
\begin{equation*}
\begin{split}
I_{\mu}+J_\mu+&J_\mu^2+\cdots+J_\mu^{p-1}=\sum_{i=0}^{p-1}{(N_{\mu}+I_{\mu})}^i=\sum_{i=0}^{p-1} \sum_{j=0}^i \binom{i}{j} N^j_\mu=\sum_{j=0}^{p-1} \sum_{i=j}^{p-1} \binom{i}{j} N^j_\mu\\ 
&=\sum_{j=0}^{p-1} \left[\binom{p}{j+1} - \sum_{i=0}^{j-1} \binom{i}{j}\right] N^j_\mu=\sum_{j=0}^{p-1}\binom{p}{j+1} N^j_\mu=N_\mu^{p-1},
\end{split}
\end{equation*}
since the characteristic of $K$ is $p$ which divides $\binom{p}{j+1}$ for every $j\in\{0,\dots ,p-2\}$. We have also used the identity $\sum_{i=0}^{p-1} \binom{i}{j}=\binom{p}{j+1}$, which comes from $\sum_{i=0}^{p-1} (1+T)^i=[(1+T)^p-1]/T$.
\end{proof}

\begin{proof}[{\it Proof of Proposition {\rm \ref{prop-dimensions1}}}]
(1) Each derivation $\delta\in\Der(A,P_n)$ is completely determined by the image $F=\delta(\epsilon_A)\in P_n$ with 
$$
F^{1+\epsilon_A+\epsilon_A^2+\cdots \epsilon_A^{p-1}}=0,
$$
which comes from (\ref{eqn-onecocycle2}); that is, $\Der(A,P_n)$ is isomorphic to the kernel of the block diagonal matrix $I_{2n-1}+E_A+E^2_A+\cdots+E^{p-1}_A$, consisting of $\lf\frac{2n-1}{p} \rf$ copies of $I_p +J_p + \cdots +J_p^{p-1}$ and $I_r +J_r + \cdots J_r^{p-1}$, where $0\leq r<p$ is the remainder of $2n-1$ modulo $p$, on the $K$-linear space $P_n$. Using Lemma \ref{sums-powerJord}, we deduce that $\dim_K\Der(A,P_n)=(p-1)\lf\frac{2n-1}{p} \rf +r=2n-1-\lf\frac{2n-1}{p} \rf$. Since $\epsilon_B=\tau \epsilon_A \tau$, where $\tau=\begin{pmatrix} 0&1\\ s&0\end{pmatrix}\in\PGL(2,K)$ is an involution, a similar argument yields $\dim_K\Der(B,P)=2n-1-\lf\frac{2n-1}{p} \rf$, as needed.

(2) Since every derivation $\delta$ on $N\cong A\ast B$ can be recovered from $\delta|_A$ and $\delta|_B$, one has the $K$-linear isomorphism $\Der(N,P_n) \rightarrow \Der(A,P_n) \times \Der(B,P_n)$.
Thus we have the exact sequence
$$
0\rightarrow\PDer(N,P_n)\rightarrow\Der(A,P_n)\times\Der(B,P_n)\rightarrow H^1(N,P_n)\rightarrow 0.
$$
On the other hand, due to Proposition \ref{Gamma-invar}, the mapping $P_n\rightarrow\PDer(N,P_n)$, which maps $F$ to the principal derivation $\delta\mapsto F^{\delta}-F$, is bijective.
Hence we have
\begin{equation*}
\begin{split}
\dim_KH^1(N,P_n)&=2(2n-1) -2 \lf \frac{2n-1}{p}\rf  -(2n-1) \\ &= 2n-1-2 \lf \frac{2n-1}{p}\rf,
\end{split}
\end{equation*}
as desired.
\end{proof}

\subsection{{\boldmath The $K[A]$-module structure of {\boldmath $\Der(\Gamma, P_n)$}}}\label{5.2}
We want to compute the action of $G\cong A\times B$ on the coholomogy group $H^1(\Gamma,P_n)$.
To this end, we first calculate $d^{\delta}$ for $d\in\Der(\Gamma,P_n)$ and $\delta\in A$ (and $\delta\in B$, as well).

Recall that the elements $e_{i,j}=[\epsilon^i_A,\epsilon^j_B]$, $i,j\in\{1,\ldots,p-1\}$, form a free basis of $\Gamma$.
For any $i,j,i',j'\in\{1,\ldots,p-1\}$ and $k\in\{0,\dots,2n-2\}$, we define $d^{(k)}_{i,j}\in\Der(\Gamma,P_n)$ by
\begin{equation} 
\label{eq:dijdef}
d^{(k)}_{i,j}(e_{i',j'})=\delta_{ii'}\delta_{jj'}b_k^{\epsilon^{-j}_B}
\end{equation}
(recall the definition of $b_k$ in \S\ref{sec-5.1}).
Then, the elements $d^{(k)}_{i,j}$, $i,j\in\{1,\ldots,p-1\}$, $k\in\{0,\dots,2n-2\}$ form a $K$-basis of $\Der(\Gamma,P_n)$.
We calculate, using equations (\ref{eqn-onecocycle}), (\ref{der-act}) and  (\ref{eq:eAactions})
\begin{equation} \label{act-on-deriv0}{}
\begin{split}
(d^{(k)}_{i,j})^{\epsilon_A}(e_{i',j'})&=[d^{(k)}_{i,j}(e_{i'+1,j'})]^{\epsilon^{j'}_B\epsilon_A\epsilon^{-j'}_B}-[d^{(k)}_{i,j}(e_{1,j'})]^{\epsilon^{j'}_B\epsilon_A\epsilon^{-j'}_B}\\ &=(\delta_{i,i'+1}-\delta_{i,1})\delta_{jj'}b^{\epsilon_A\epsilon^{-j}_B}_k.
\end{split}
\end{equation}
Here we have set $e_{0,j}=1$ for convenience. By computation we have
\begin{equation} \label{act-onderiv}
\left(d_{i,j}^{(k)}\right)^{\epsilon_A}=
\begin{cases}
d_{i-1,j}^{(k)}+d_{i-1,j}^{(k-1)} & \mbox{ if } p \nmid k \mbox{ and } i \neq 1, \\
d_{i-1,j}^{(k)} & \mbox{ if } p\mid k \mbox{ and } i\neq 1, \\
-\sum_{i'= 1}^{p-1} d_{i',j}^{(k)}+ d_{i',j}^{(k-1)} & \mbox{ if } p \nmid k \mbox{ and } i=1, \\
-\sum_{i'=1}^{p-1} d_{i',j}^{(k)} & \mbox{ if } p \mid k \mbox{ and } i=1.
\end{cases}
\end{equation}

From this one can compute the matrix for $\epsilon_A$ by means of the basis $\{d^{(k)}_{ij}\}$, ordered lexicographically 
$$
d^{(0)}_{1,1},d^{(1)}_{1,1},\ldots,d^{(2(n-1))}_{1,1},d^{(0)}_{2,1},d^{(1)}_{2,1},\ldots,d^{(2(n-1))}_{2,1},\ldots.
$$
The matrix $Q_A$ in question is the square matrix of degree $(2n-1)(p-1)^2$, which is first of all a block diagonal
\begin{equation} \label{QA-def}
Q_A=\mathrm{diag}(N,\ldots,N),
\end{equation}
consisting of $p-1$ copies (indexed by $j=1,\ldots,p-1$) of a square matrix $N$ of degree $(2n-1)(p-1)$, and the matrix $N$ is the `tensor product' of $M$ in (\ref{protype}) and $E_A$ in (\ref{eqn-EA}), i.e., the matrix obtained by replacing $\pm 1$ in $M$ with $\pm E_A$.


In more algebraic terms, the $K[A]$-module structure of $\Der(\Gamma,P_n)$ is described as follows.
As in Theorem \ref{theoremnis1} (1), let $L$ be the free $\Z$-module of rank $p-1$ with the $\Z[A]$-module structure given by $\epsilon_A\mapsto M$ (with respect to some basis $\{v_1,\ldots,v_{p-1}\}$), and $W$ be the $K$-vector space of dimension $2n-1$ with the $K[A]$-module structure by $\epsilon_A\mapsto E_A$ (with respect to some basis $\{w_1,\ldots,w_{2n-1}\}$).
As in \S\ref{3.1}, we simply denote by $J_{\mu}$ the $K$-vector space $K^{\mu}$ with the $K[A]$-module structure by $\epsilon_A\mapsto J_{\mu}$, we have
\begin{equation}\label{eqn-W}
W\cong J_p^{\lf \frac{2n-1}{p}\rf} \oplus J_r
\end{equation}
as $K[A]$-module, where $r$ is the remainder of $2n-1$ modulo $p$; here, as before, we put $J_0=\{0\}$ for convenience.
Notice that, as we have seen in \S\ref{sec-5.1}, we have $P_n\cong W$ as $K[A]$-module.
Then, what we have shown above amounts to the $K[A]$-module isomorphism
\begin{equation}\label{eqn-WL}
\Der(\Gamma,P_n)\cong (W\otimes_{\Z}L)^{p-1}.
\end{equation}

To count the number of indecomposable summands, let us consider a slightly general situation as follows: 
Let $U$ be a right free $K[A]$-module, and consider the $K[A]$-module $U\otimes_{\Z}L$.
Any element $x\in U\otimes_{\Z}L$ can be written as $x=\sum^{p-1}_{i=1}a_i\otimes v_i$, where $a_i\in U$ ($1\leq i\leq p-1$).
We have
$$
x^{\epsilon_A}=\sum^{p-1}_{i=1} a^{\epsilon_A}_i\otimes v^{\epsilon_A}_i= a_1^{\epsilon_A}\otimes \sum_{i=1}^{p-1}(-v_i) +\sum_{i=2}^{p-1}a^{\epsilon_A}_i\otimes v_{i-1}=\sum^{p-2}_{i=1}(a^{\epsilon_A}_{i+1}-a^{\epsilon_A}_1)\otimes v_i-a^{\epsilon_A}_1\otimes v_{p-1}.
$$
By straightforward calculations, we see that the equation $x=x^{\epsilon_A}$ holds if and only if 
\begin{equation*}a_i-a_{i+1}^{\epsilon_A} + a_1^{\epsilon_A}=0\ (i=1,\ldots,p-2)\quad\textrm{and}\quad a_{p-1}+a_1^{\epsilon_A}=0,
\end{equation*}
if and only if
\begin{equation}\label{eqn-condition12}
a_i=a_1^{1+\epsilon^{-1}_A+\cdots+\epsilon^{1-i}_A}\ (i=1,\ldots,p-1)\quad\textrm{and}\quad a_1^{1+\epsilon^{-1}_A+\cdots+\epsilon^{1-p}_A}=0.
\end{equation}
Hence each $x\in(U\otimes_{\Z}L)^A$ is determined by its first coefficient $a_1$, which is further subject to the second condition in (\ref{eqn-condition12}).

If $U=J_{\mu}$ ($1\leq\mu\leq p$), then by a similar calculation to that in Lemma \ref{sums-powerJord}, we deduce that the second equality of (\ref{eqn-condition12}) gives a non-trivial condition only when $\mu=p$, and that
\begin{equation}\label{eqn-dimcount}
\dim_K(J_{\mu}\otimes_{\Z}L)^A=\begin{cases}\mu &\textrm{if $\mu< p$,}\\ p-1 &\textrm{if $\mu=p$.}\end{cases}
\end{equation}

\begin{proposition} \label{24}
{\rm (1)} As a $K[A]$-module, we have 
$$
J_p\otimes_{\Z}L\cong J^{p-1}_p.
$$

{\rm (2)} For $1\leq r \leq p-1$, we have
$$
J_r\otimes_{\Z}L\cong J^{r-1}_p\oplus J_{p-r}
$$
as a $K[A]$-module.
\end{proposition}

\begin{proof}
(1) First notice that any direct summand has to be of the form $J_{\mu}$ with $1\leq\mu\leq p$ (cf.\ \S\ref{3.1}); in particular, it is of dimension at most $p$.
Since the number of indecomposable summands of $J_p\otimes_{\Z}L$ is $p-1$, and since $\dim_KJ_p\otimes_{\Z}L=p(p-1)$, it has only $J_p$ as its direct summands.

(2) Consider the obvious exact sequence
$$
0\rightarrow J_r\rightarrow J_p\rightarrow J_{p-r}\rightarrow 0.
$$
Tensoring the free $\Z$-module $L$ yields the exact sequence
\[
 0 \rightarrow J_r \otimes L\rightarrow  J_p\otimes L \rightarrow J_{p-r}\otimes L \rightarrow 0,
 \]
from which we obtain 
\[
 0 \rightarrow (J_r \otimes L)^A {\rightarrow}
  (J_p\otimes L)^A  {\rightarrow} (J_{p-r} \otimes L)^A  \rightarrow H^1(A,J_r \otimes L) \rightarrow  H^1(A,J_p\otimes L).
\]
Since $J_p\otimes L\cong J^{p-1}_p$ is a free $K[A]$-module, we have $H^1(A,J_p\otimes L)=\{0\}$.
By (\ref{eqn-dimcount}), we know that $\dim_KH^1(A,J_r\otimes L)=1$.
Now if $J_r \otimes L\cong\bigoplus^m_{i=1}J_{\mu_i}$, we have 
\[
{\textstyle H^1(A,J_r \otimes L)\cong\bigoplus^m_{i=1} H^1(A,J_{\mu_i}).}
\]
Since $H^1(A,J_{\mu})$ is zero for $\mu=p$, and is $1$-dimensional for $1\leq\mu<p$, $J_r\otimes_{\Z}L\cong J^{r-1}_p\oplus J_{p-r}$ is the only possibility, for $\dim_KJ_r \otimes L=r(p-1)$.
\end{proof}

\begin{remark} We note that the results of Proposition \ref{24} can also be deduced from the case I.a) in the multiplication table of \cite{Almkvist}. Another proof of  Proposition \ref{24} (1) can be found in \cite[Lemma 7.1.13]{Campbell}. Here, we have provided a different method of proof. 
\end{remark}

\begin{remark}
The problem determining the indecomposable summands of tensor products in representation theory is called the Clebsch-Gordan problem. The case of Jordan normal form in modular representation theory was studied by many authors (see \cite{GlasbyCheryl} and references within). 
\end{remark}

\begin{corollary}\label{22}
 The $K[A]$-module structure of $\mathrm{Der}(\Gamma,P_n)$ 
is given by
\[
 \mathrm{Der}(\Gamma,P_n)\cong
 \begin{cases}
\left( J_{p}^{(p-1)\lf \frac{2n-1}{p}\rf} \oplus J_p^{r-1} \oplus J_{p-r} \right)^{p-1} & \mbox{ if } p\nmid 2n-1, \\
 \left( J_{p}^{(p-1) \frac{2n-1}{p}} \right)^{p-1} & \mbox{ if } p\mid 2n-1.
\end{cases}
\]
Here, in the first case, $r$ denotes the remainder of $2n-1$ modulo $p$.
\end{corollary}

\begin{proof}
This follows from (\ref{eqn-W}), (\ref{eqn-WL}), and Proposition \ref{24}.
\end{proof}

\section{Computations on cohomology}
\label{6}

In this section we focus on the computation of both the $K[A]$ and  
$K[A\times B] =K[N/\Gamma]=K\big[(A*B) /[A,B]\big]$-module structure 
of $H^1(\Gamma,P_n)$. We will give two different proofs of Theorem \ref{mainFree} (1).

\subsection{First Proof of Theorem \ref{mainFree} (1)}
In order to form the quotient we need to know exactly how the module of principal derivations sits inside the module of derivations. 

\begin{definition}
From now on we will denote by $h$ the endomorphism $h=\epsilon_A-1$. 
Let  $V$ be  a $K[A]$-module. We will say that an element $w\in V$ has order $\ord(w)=a$ if 
$w\in \ker h^a - \ker h^{a-1}$.  We will say that  $u$ generates a module $M$ isomorphic to  $J_a$ as 
a $K[A]$-module, if and only if the set $\{u,h(u),\ldots,h^{a-1}(u)\}$ forms a $K$-vector space basis of $M$. Notice that the generator $u$ of a module isomorphic to $J_a$ has order $a$. 
\end{definition}

Generators of direct summands $J_p$ of the space of principal derivations have order $p$ and therefore go to generators of $J_p$ summands of the space $\mathrm{Der}(\Gamma,P_n)$. Let $b_r=(T^p-T)^{\lf \frac{2n-1}{p}\rf}\binom{T}{r}$ be a generator of the $J_r$ component  of $P_n$. Let $\psi$ be the map sending $P_n \ni b$ to the principal derivation $\gamma \mapsto b^\gamma -b$. For $N_1$ and $N_2$ given by Corolllary \ref{22}, the principal derivation $\psi(b_r)$ is then decomposed as a sum \begin{equation} \label{psiexp1}\psi(b_r)= \sum_{i=1}^{N_1} a_i + \sum_{j=1}^{N_2} \beta_j \in J_p^{N_1} \oplus J_{p-r}^{N_2},\end{equation}where $a_i \in J_p$ and $b_j \in J_{p-r}$.  Since the order of $b_r$ is $r$, there is at least one summand of order $r$ and all other summands have order $\geq r$. It is clear that if $a_i$ in $J_p$ has order $t \geq r$, then there is an element $a_i' \in J_p$ such that $h^{p-t}(a_i')=a_i$.  We will prove that the elements  $\beta_j$ in eq. (\ref{psiexp1}) can also be written as $h^{p-r}(\beta_j')$. Notice that if $r=1$ this means that $\beta_j=0$, since $J_r \otimes J_{p-1} \cong J_p^{r-1} \oplus J_{p-r}$, and if $r=1$, then there is no $J_p$ direct summand in $J_1 \otimes J_{p-1}\cong J_{p-1}=J_{p-r}$.

\begin{proposition} \label{rmeg1}
The direct summand $J_r$ of $P_n$ given in eq. (\ref{eqn-EA}) is mapped  inside a direct summand of $\mathrm{Der}(\Gamma,P_n)$, isomorphic to $J_p^N$ for some $N \in \mathbb{N}$.
\end{proposition}

%
%
%

In order to prove of proposition \ref{rmeg1} we have to introduce a combinatorial point of view of the basis of $J_r \otimes J_{p-1}$. 
Consider a basis $b_1,\ldots,b_r$ of the module  $J_r$ such that 
$\epsilon_A(b_\nu)=b_\nu + b_{\nu-1}$ and a basis $\epsilon_1,\ldots,\epsilon_{p-1}$ of $J_{p-1}$, such that $\epsilon_A (\epsilon_i) =\epsilon_i+ \epsilon_{i-1}$. 
Also  $b_i$ and $\epsilon_i$ are considered to be zero for $i \leq 0$. 
The elements $\epsilon_{i,j}=b_i \otimes \epsilon_j$ form a basis for the space $J_r \otimes J_{p-1}$. 
Geometrically we consider the elements $\epsilon_{i,j}$ to form an $r \times (p-1)$ grid arranged as follows:
\begin{equation} \label{grid}
\begin{array}{ccccc}
e_{1,p-1} & e_{2,p-1} & \cdots &  e_{r-1,p-1} & e_{r,p-1} \\
e_{1,p-2} & e_{2,p-2} & \cdots &  e_{r-1,p-2} & e_{r,p-2} \\
\vdots &  \vdots & & \vdots & \vdots \\
e_{1,1} & e_{2,1} & \cdots & e_{r-1,1} & e_{r,1} 
\end{array}
\end{equation}
One can approach the action of $\epsilon_A$ on basis elements of 
$J_r \otimes J_{p-1}$ in the following way: 

\noindent 
First we compute the effect of $\epsilon_A$:
\[
\epsilon_A(e_{i,j})=(b_i + b_{i-1}) \otimes (\epsilon_{j} +
\epsilon_{j-1}).
\]
Next, we notice that  the action of a $\lambda$-power of $\epsilon_A$ is given by:  
\[
\epsilon_A^\lambda ( e_{i,j} )= 
\left( \sum_{\nu=0}^\lambda \binom{\lambda}{\nu} b_{i-\nu } \right)
\otimes
\left(
 \sum_{\mu=0}^\lambda \binom{\lambda}{\mu} \epsilon_{j-\mu }
 \right) 
\]
Finally we compute:
\[
h^k=(\epsilon_A-1)^k=\sum_{\lambda=0}^k \binom{k}{\lambda} 
(-1)^{k-\lambda} \epsilon_A^{\lambda}.
\]
Thus the action of $h^k$  on $e_{ij}$ is given by:
\begin{eqnarray}
h^k(e_{i,j}) & = & (\epsilon_A-1)^k (e_{i,j})  \nonumber \\
& =& \sum_{\lambda=0}^k \binom{k}{\lambda} 
(-1)^{k-\lambda} \epsilon_A^{\lambda} (e_{i,j})  \nonumber \\
& = & \sum_{\lambda=0}^k \binom{k}{\lambda} 
(-1)^{k-\lambda}
 \sum_{\nu=0}^\lambda \binom{\lambda}{\nu}
\sum_{\mu=0}^\lambda \binom{\lambda}{\mu} e_{i-\nu,j-\mu} \nonumber  \\
& = & \sum_{\lambda=0}^k \binom{k}{\lambda} 
(-1)^{k-\lambda}
 \sum_{\nu=0}^k \binom{\lambda}{\nu}
\sum_{\mu=0}^k \binom{\lambda}{\mu} e_{i-\nu,j-\mu} \nonumber \\
& =& \sum_{\nu=0}^i \sum_{\mu=0}^{j} 
\left( 
\sum_{\lambda=0}^k (-1)^{k-\lambda}
\binom{\lambda}{\mu}
\binom{\lambda}{\nu}
\binom{k}{\lambda}
\right)
e_{i-\nu,j-\mu}. \label{exptok}
\end{eqnarray}
In eq. (\ref{exptok}) above, we have extended the summation up to $k$
since $\binom{I}{J}=0$ for $J>I$.  
We define
\[
\delta^k_{\nu,\mu}:=
\sum_{\lambda=0}^k (-1)^{k-\lambda}
\binom{\lambda}{\mu}
\binom{\lambda}{\nu}
\binom{k}{\lambda} 
\]
%
%
%
\begin{lemma}\label{delta-comp} We have
\begin{equation*}
\delta_{\nu,\mu}^k  = 
\binom{k}{\nu} \binom{\nu}{\mu-k+\nu} = 
\binom{k}{\nu} \binom{\nu}{k-\mu }.
\end{equation*}
\end{lemma}
\begin{proof} Indeed, using \cite[Equation 3.49]{Gould} we get
	\begin{align*}\delta^k_{\nu,\mu}& =
	\sum_{\lambda=0}^k (-1)^{k-\lambda}
	\binom{k}{\lambda} 
	\binom{\lambda}{\nu}
	\binom{\lambda}{\mu}
	 = \sum_{\lambda=0}^k (-1)^{k-\lambda} 
	\binom{k}{\nu}
	\binom{k-\nu}{k-\lambda}
	\binom{\lambda}{\mu} \\
&= \binom{k}{\nu}\sum_{\lambda=0}^k (-1)^{k-\lambda}
	\binom{k-\nu}{k-\lambda}
	\binom{\lambda}{\mu} 
	= \binom{k}{\nu}\sum_{i=0}^{k-\nu}{(-1)}^i \binom{k-\nu}{i}\binom{k-i}{\mu} \\
&= \binom{k}{\nu} \binom{\nu}{\mu-k+\nu},
\end{align*}as needed.
\end{proof}
Combining Lemma \ref{delta-comp} with eq. (\ref{exptok}) we deduce that
\begin{equation}\label{ha-ypol} \boxed{
h^k(e_{i,j})=
\sum_{\nu=0}^i \sum_{\mu=0}^{j} 
\delta_{\nu,\mu}^k
e_{i-\nu,j-\mu}=
\sum_{\nu=0}^i \sum_{\mu=0}^{j} 
\binom{k}{\nu} \binom{\nu}{k-\mu }
e_{i-\nu,j-\mu}.
}
\end{equation}

We will follow a different approach for computing the coefficients 
of the basis elements which appear in the expansion of $h^k (e_{i,j})$. 
Notice first that 
\[
h(e_{i,j})= e_{i-1,j} + e_{i,j-1} + e_{i-1,j-1}. 
\]
This means that an element $e_{i,j}$ is moved by $h$ to the sum 
of three elements in the grid of equation (\ref{grid}) which lie 
just below to the right and right and below. 
\begin{equation} \label{3dir}
\xymatrix@C=0.4cm{
e_{i-1,j} & e_{i,j} \ar[d] \ar[l] \ar[dl] \\
e_{i-1,j-1} & e_{i,j-1} 
} 
\end{equation}
Assume that all the above arrows have length $1$. 
For the $h^2(e_{i,j})$ we have 
\[
\xymatrix@C=0.4cm{
e_{i-2,j} & e_{i-1,j} \ar[d] \ar[l] \ar[dl] & e_{i,j} \ar[d] \ar[l] \ar[dl] \\
e_{i-2,j-1} & e_{i-1,j-1} \ar[d] \ar[l] \ar[dl] & e_{i,j-1} \ar[d] \ar[l] \ar[dl] \\
e_{i-2,j-2} & e_{i-2,j-2} & e_{i,j-2}
} 
\]
On the other hand side we compute 
\begin{eqnarray*}
h^2(e_{i,j}) &= & h(e_{i-1,j}) + h(e_{i,j-1}) + h(e_{i-1,j-1}) \\
&= & (e_{i-2,j} + e_{i-1,j-1} + e_{i-2,j-1}) + \\
 &  & + (e_{i-1,j-1} + e_{i,j-2} + e_{i-1,j-2}) \\ 
 &  & + (e_{i-2,j-1} + e_{i-1,j-2} + e_{i-2,j-2}) \\
 & = &  e_{i-2,j} + 2 e_{i-1, j-1} +2 e_{i-2,j-1} + e_{i,j-2} + 
 2e_{i-1,j-2}.
\end{eqnarray*}
By induction we can prove the following 
\begin{lemma} \label{countPaths} The coefficient of $e_{\nu,\mu}$ in 
$h^k(e_{i,j})$ is the number of paths of length $k$ we can form from 
$e_{i,j}$ to $e_{\nu,\mu}$ if from each intermediate node we can go in three
directions of eq. (\ref{3dir}).
\end{lemma}

\begin{lemma} \label{s-line-move}
Fix an $i,j$, where $i\in\{1,\dots ,r\}$ and $j\in\{1,\dots ,p-1\}$.
The image of $h^{k}(e_{i,j})$ is contained in the space generated by vertices in the grid of eq. (\ref{grid}), which  lie left and down of the line $s+t=i+j-k$, i.e. it is contained in the vector space generated by the elements $e_{s,t}$ where $s+t \leq i+j-k$. 
\end{lemma}
\begin{proof}
This is immediate by either the geometric viewpoint explained in Lemma \ref{countPaths} or by the computation given in  eq. (\ref{ha-ypol}).
\end{proof}

\begin{proposition}\label{cor29}The kernel of $h^{k}$ is the vector space generated by the elements $e_{i,j}$ where $i+j\leq k+1$.
\end{proposition}
\begin{proof} 
	Consider an element $e_{s,t}$ in the expression of $h^k(e_{i,j})$ in eq. (\ref{ha-ypol}). It follows from Lemma \ref{s-line-move} that $s+t\leq i+j -k\leq 1$ and hence either $s\leq 0$ or $t\leq 0$. Thus, $e_{s,t}$ is the zero element and $h^k(e_{i,j})=0$, as needed.
	Conversely, if $e_{i,j}$ is a basis element such that $i+j\geq k+2$, then there exists always a lattice path of lenght $k$, where the allowed moves are down and left (see figure \ref{fig:LatticePath}), which sends $e_{i,j}$ to a non-zero basis element $e_{s,t}$. It follows that $h^k(e_{i,j})\neq 0$ and hence $e_{i,j}$ does not belong to the kernel of $h^{k}$.
\begin{figure}[h]
\begin{tikzpicture}[scale=0.7, transform shape]

\def\dx{1.5} 
\def\dy{1.5} 

\draw[very thick,red] (0*\dx, 3*\dy) -- (3*\dx, 0*\dy) node[right, below] {$\ell:i+j=k+1$};
\node at (\dx,0.5*\dx) {$i+j \leq k+1$};

\foreach \x in {0,...,3} {
    \foreach \y in {0,...,4} {
                \fill[black] (\x*\dx, \y*\dy) circle (3pt);
    }
}

\node[circle, blue!50] (A) at (3*\dx, 4*\dy) {};

\node[circle] (B) at (2*\dx, 4*\dy) {};
\node[circle] (C) at (2*\dx, 3*\dy) {};
\node[circle] (D) at (3*\dx, 3*\dy) {};

\node[circle] (F) at (2*\dx, 0*\dy) {};
\node[circle] (E) at (3*\dx, 1*\dy) {};
\node[circle] (K) at (3*\dx, 0*\dy) {};
\node[circle] (KK) at (\dx, 0*\dy) {};

\draw[->, dotted,thick] (A) -- (B);
\draw[->, dotted,thick] (A) -- (C);
\draw[->, dotted,thick] (A) -- (D);

\draw[->,blue,thick] (C) -- (F); 
\draw[->,blue,thick] (E) -- (K) -- (KK);

\pgfdeclareverticalshading{greenblue}{\dy*4}{
    color(0pt)=(green);
    color(\dy*4)=(blue)
}

 \fill[shading = axis,rectangle, left color=blue, right color=yellow,shading angle=30, opacity=0.1] (0,0) -- (0*\dx, 3*\dy) -- (3*\dx, 0*\dy) -- cycle;


\end{tikzpicture}

\caption{ \label{fig:LatticePath} Two lattice paths in blue color} 

\end{figure}
\end{proof}

\begin{figure}
\scalebox{0.5}{
\begin{tikzpicture}[scale=1]

\def\dx{1.5} 
\def\dy{1.5} 
\def\P{11} 
\def\r{4} 

\draw[very thick,red] (0*\dx, 3*\dy) -- (3*\dx, 0*\dy) node[right, below] {$\ell_1:i+j=r+1$};
\node at (1.2*\dx,0.5*\dy) {$\mathrm{ker} h^r: i+j \leq r+1$};
\node at (1.6*\dx,3.5*\dy) {$\mathrm{ker} h^{p-r}: i+j \leq p-r+1$};

\draw[very thick,red ] (0*\dx, 7*\dy) -- (3*\dx, 4*\dy) node[near start, above right] {$\ell_2:i+j=p-r+1$};

\foreach \x in {0,...,3} {
    \foreach \y in {0,...,8} {
                \fill[black] (\x*\dx, \y*\dy) circle (3pt);
    }
}

\node[circle, blue!50] (A) at (3*\dx, 4*\dy) {};

\node[circle] (B) at (2*\dx, 4*\dy) {};
\node[circle] (C) at (2*\dx, 3*\dy) {};
\node[circle] (D) at (3*\dx, 3*\dy) {};

\node[circle] (F) at (2*\dx, 0*\dy) {};
\node[circle] (E) at (3*\dx, 1*\dy) {};
\node[circle] (K) at (3*\dx, 0*\dy) {};
\node[circle] (KK) at (\dx, 0*\dy) {};




 \fill[shading = axis,rectangle, left color=blue, right color=yellow,shading angle=30, opacity=0.15] (0,0) -- (0*\dx, 3*\dy) -- (3*\dx, 0*\dy) -- cycle;

\fill[shading = axis,rectangle, left color=green, right color=yellow,shading angle=30, opacity=0.07] (0,0) -- (0, 7*\dy) --(1*\dx, 6*\dy) -- (3*\dx, 4*\dy) --  (3*\dx,0) -- cycle;

\draw[decorate,decoration={brace,amplitude=10pt},thick] (-\dx, 0) -- ( -\dx, 8*\dy) node[midway, left] {$p-1 \;\;\;$};
\draw[decorate,decoration={brace,amplitude=10pt,mirror},thick] (0*\dx, -0.5*\dy) -- (3*\dx, -0.5*\dy) node[midway, below=10pt] {$r$};
\draw[decorate,decoration={brace,amplitude=10pt},thick] (3.2*\dx, 4*\dy) -- (3.2*\dx, 1*\dy) node[midway, above=2pt, right=10pt] {$p-2r$};
\draw[decorate,decoration={brace,amplitude=10pt},thick] (4.1*\dx, 7*\dy) -- (4.1*\dx, 0*\dy) node[midway, right=8pt] {$p-r$};
\draw[decorate,decoration={brace,amplitude=10pt},thick] (-0.3*\dx, 0*\dy) -- (-0.3*\dx, 3*\dy) node[midway, left=10pt] {$r$};
\draw[decorate,decoration={brace,amplitude=10pt},thick] (-0.3*\dx, 4*\dy) -- (-0.3*\dx, 8*\dy) node[midway, above= 23pt, left=16pt, rotate=90] {$p-1-r$};

\end{tikzpicture}
}
\scalebox{0.5}{
\begin{tikzpicture}[scale=1]

\def\dx{1.5} 
\def\dy{1.5} 

\draw[very thick,red] (0*\dx, 6*\dy) -- (6*\dx, 0*\dy) node[midway, above right] {$\ell_1:i+j=r+1$};
\node at (1.2*\dx,3.5*\dy) {$\mathrm{ker} h^r: i+j \leq r+1$};
\node at (1.6*\dx,0.5*\dy) {$\mathrm{ker} h^{p-r}: i+j \leq p-r+1$};

\draw[very thick,red ] (0*\dx, 3*\dy) -- (3*\dx, 0*\dy) node[right, below] {$\ell_2:i+j=p-r+1$};

\foreach \x in {0,...,6} {
    \foreach \y in {0,...,7} {
                \fill[black] (\x*\dx, \y*\dy) circle (3pt);
    }
}

\node[circle, blue!50] (A) at (3*\dx, 4*\dy) {};

\node[circle] (B) at (2*\dx, 4*\dy) {};
\node[circle] (C) at (2*\dx, 3*\dy) {};
\node[circle] (D) at (3*\dx, 3*\dy) {};

\node[circle] (F) at (2*\dx, 0*\dy) {};
\node[circle] (E) at (3*\dx, 1*\dy) {};
\node[circle] (K) at (3*\dx, 0*\dy) {};
\node[circle] (KK) at (\dx, 0*\dy) {};




 \fill[shading = axis,rectangle, left color=blue, right color=yellow,shading angle=30, opacity=0.07] (0,0) -- (0*\dx, 6*\dy) -- (6*\dx, 0*\dy) -- cycle;

 \fill[shading = axis,rectangle, left color=blue, right color=yellow,shading angle=30, opacity=0.07] (0,0) -- (0*\dx, 3*\dy) -- (3*\dx, 0*\dy) -- cycle;


\draw[decorate,decoration={brace,amplitude=10pt},thick] (-0.8*\dx, 0) -- ( -0.8*\dx, 7*\dy) node[midway, left] {$p-1 \;\;\;$};
\draw[decorate,decoration={brace,amplitude=10pt,mirror},thick] (0*\dx, -0.5*\dy) -- (6*\dx, -0.5*\dy) node[midway, below=10pt] {$r$};
\draw[decorate,decoration={brace,amplitude=10pt},thick] (6.2*\dx, 3*\dy) -- (6.2*\dx, 0*\dy) node[midway, right=10pt] {$p-r$};
\draw[decorate,decoration={brace,amplitude=10pt},thick] (-0.3*\dx, 0*\dy) -- (-0.3*\dx, 6*\dy) node[midway, left=10pt] {$r$};

\end{tikzpicture}
}

\caption{$J_r \otimes J_{p-1}$ for $r<p-r$ (left) and $r>p-r$ (right). \label{JrJp}}

\end{figure}
Consider again the map $\psi$ sending an element $x \in P_n$ to the principal derivation 
$\gamma \mapsto x^\gamma-x$. 

We would like to prove that the element $b_{2n-2}\in P_n$ of order $r$ generating the direct summand $J_r$ of $P_n$ is mapped to  $\psi(b_{2n-2})=h^{p-r}(y)$ for some element $y\in J_r\otimes J_{p-1}$. 


Since the map $\psi$ is $h$-equivariant,  it follows that $\psi(b_{2n-2})$ lies in the kernel of $h^r$, which is the space generated by elements left and below of line $\ell_1$, given by vertices $e_{i,j}$ such that $i+j=r+1$. 
It follows from Proposition \ref{cor29} that $\textrm{ker}(h^{r})$ is the vector space generated by vertices in the $r\times r$ triangle in the lower left corner of the $r\times (p-1)$ grid, including those on the line $\ell_1$ (see Figure 2) and we will prove that $\textrm{ker}(h^{r})\subseteq \textrm{Im}(h^{p-r})$. It suffices to show that $\textrm{dimker}(h^r)\leq \textrm{dimIm}(h^{p-r})$.
Moreover, the kernel of the map ${h}^{p-r}$ consists of elements which lie left and down of line $\ell_2$, including those on the line $\ell_2$, where $\ell_2$ is the line consisted by elements $e_{i,j}$ such that $i+j=p-r+1$. We distinguish the following cases:

{\bf Case 1:} Suppose that $1< r<p-r$. Then, the kernel of $h^{p-r}$ has dimension:
\[
(p-r)+(p-r-1) +\cdots + (p-2r+1) =(p-2r)\cdot r + \frac{r(r+1)}{2}
\] (see Figure 2).
Consequently we have $$\dim \mathrm{Im}({h}^{p-r})=r(p-1)-\dim \mathrm{ker}({h}^{p-r})=\frac{3(r^2-r)}{2}$$ Since $\dim \mathrm{ker}({h}^{r})=1+2+\cdots+r=\frac{r(r+1)}{2}$ and $r\geq 2$, we deduce that $\dim \mathrm{Im}({h}^{p-r})\geq \dim \mathrm{ker}({h}^{r})$, as needed.


{\bf Case 2:} Suppose that $p-r<r$.
This case differs from the previous one, since here the line $\ell_2$ is to the left of the line $\ell_1$. 
 The kernel of the map ${h}^{p-r}$ is generated by the vertices of the $(p-r) \times (p-r)$ triangle of points which are below and to the left of line $\ell_2$, including those on $\ell_2$, which is given by equation $i+j=p-r+1$, and has dimension $1+2+\cdots+ (p-r)=(p-r)(p-r+1)/2$ (see Proposition \ref{cor29}). 
It follows that 
$$
\dim \mathrm{Im}({h}^{p-r})=r(p-1)-\dim \mathrm{ker}({h}^{p-r})=\frac{4rp-r-r^2-p^2-p}{2}.
$$
 Similarly, the kernel of the map ${h}^{r}$ is generated by the vertices of the $r \times r$ triangle of points which are below and to the left of line $\ell_1$, including those on $\ell_1$, given by equation $i+j=r+1$, and has dimension $1+2+\cdots+ r=r(r+1)/2$ (see Proposition \ref{cor29}). Therefore, 
 \begin{equation}\label{eq28}
	\dim \mathrm{Im}({h}^{p-r})-\dim \mathrm{ker}({h}^{r})=-r^2+(2p-1)r-\frac{p^2+p}{2}.
\end{equation} 
Since $r\leq p-1$, the condition $p-r<r$ yields $p\geq 3$. Moreover, if $p=3$ then $r=2$ and hence $\dim \mathrm{Im}({h}^{p-r})=\dim \mathrm{ker}({h}^{r})=3$. We assume now that $p\geq 5$. Then, the polynomial in equation (\ref{eq28}) has two real roots $$\rho_{1}=\frac{2p-1-\sqrt{2p^2-6p+1}}{2},\,\rho_{2}=\frac{2p-1+\sqrt{2p^2-6p+1}}{2},$$ where $\rho_1< \frac{p}{2}$ and $\rho_2>p-1$. Since $\frac{p}{2}<r\leq p-1$, we conclude that $\dim \mathrm{Im}({h}^{p-r})>\dim \mathrm{ker}({h}^{r})$, as needed.  
  



{\bf Case 3} Suppose that $r=1$. We will treat this special case differently.
Recall that we have the basis $b_0,\ldots,b_{2n-2}$ for the space $P_n$ as defined in eq. (\ref{bk}) and the elments $d_{i,j}^{(k)}$, for $0\leq k \leq 2(n-1)$, $1\leq i,j \leq p-1$ as a basis for $\mathrm{Der}(\Gamma,P_n)$ as given in eq. (\ref{eq:dijdef}). 

Since $2n-1$ has remainder $r=1$ when divided by $p$, we have that $2(n-1) =p q$ is divided by $p$ and
$b_{2(n-1)}=b_{pq}=(T^p-T)^q$ according to eq. (\ref{bk}).

 Consider the principal derivation $\psi(b_{pq})$ sending $\gamma_{i,j}=[e_A^i,e_B^j]\mapsto b_{pq}^{\gamma_{i,j}}-b_{pq}$. 
 We have
\begin{equation}
 \label{pdev-1}
  b_{pq}^{\epsilon_A^i \epsilon_B^j \epsilon_A^{-i} \epsilon_B^{-j}} 
  -b_{pq}=
  \sum_{\lambda=0}^{pq} a(i,j)^{(pq)}_\lambda b_\lambda^{e_B^{-j}}.
\end{equation} 
By applying $e_B^j$ in eq. (\ref{pdev-1}), using that $b_{2n-2}$ is $\epsilon_A$-invariant we obtain 
\begin{equation} 
\label{r1case1} 
b_{pq}^{\epsilon_B^j \epsilon_A^{-i}} -b_{pq}^{\epsilon_B^j}
=\sum_{\lambda=0}^{pq} a(i,j)^{(pq)}_\lambda b_\lambda. 
\end{equation}
Using the matrix form for $e_B$ given in eq. (\ref{matrices}) and the action on polynomials as given in eq. (\ref{pol-action}) we compute  
\[
e_B^j =
\begin{pmatrix}
1 & 0 \\
js & 1
\end{pmatrix}
\]
and
\begin{align*}
b_{pq}^{e_B^j} &=(sjT+1)^{pq} \frac{(1- (sjT+1)^{p-1})^q}{ (1+sjT)^{pq}}
\\
&=
\left(
\sum_{\nu=1}^{p-1} \binom{p-1}{\nu} (sjT)^\nu
\right)^q.
\end{align*}
Also $b_{pq}^{e_B^j e_A^{-i}}$ is again a polynomial of degree at most $(p-1)q$. 
The left hand side of equation (\ref{r1case1}) is a polynomial of degree at most $(p-1)q$, that is it cannot  involve elements $b_{pq}$ as summands, i.e. 
$a(i,j)^{(pq)}_{pq}=0$. Thus the expression $\psi(b_{pq})$ does not involve basis elements $d_{i,j}^{(pq)}$,  which give rise to the $J_{p-1}$ summands in $\mathrm{Der}(\Gamma,P_n)$.  

\smallskip


\medskip
We have seen that $\psi(b_{2n-2})={h}^{p-r}(y)$ for some $y$ of order $p$, therefore $\psi(b_{2n-2})$ is inside a $J_p$ direct summand of $\mathrm{Der}(\Gamma,P_n)$.
Indeed, the element $y$ can be expressed as a linear combination
\[
y =\sum_{i=1}^{N_1} \lambda_i a_i +\sum_{j=1}^{N_2} \mu_j b_j, 
\]
where $a_i$ are generators of $J_p$ summands and $b_j$ are generators of $J_r$ summands.
Since $y$ has order $p$ at least one of the coefficients $\lambda_{i_0}$ is not zero, and by basis exchange lemma of linar algebra we see that $\{a_1,\ldots,a_{i_0-1},y,a_{i_0+1},\ldots,\}$ are also generators of the $J_p$ summands.

 We can now proceed to the computation of the quotient  
$\mathrm{Der}(\Gamma,P_n)/\mathrm{PDer}(\Gamma,P_n)$ is isomorphic to 
\begin{align}
H^1(\Gamma,P_n)&\cong \Der(\Gamma,P_n)/\mathrm{PDer}(\Gamma,P_n)
 \nonumber\\&\cong  J_p^{(p-1) \lf \frac{(2n-1)(p-1)}{p} \rf-1 
 - \lf \frac{2n-1}{p} \rf
 } \oplus J_p/J_r  \oplus  J_{p-r}^{p-1}
\nonumber \\
& \cong 
 J_p^{(p-1)(2n-1) - p\lc \frac{2n-1}{p}\rc}   \oplus  J_{p-r}^{p}.  \label{coho-A}
\end{align}


\subsection{Second Proof of Theorem \ref{mainFree} (1) }
Our second proof of Theorem \ref{mainFree} (1) uses the algebraic  theory of  curves $(x^p-x)(y^p-y)=c$.

Recall that the Artin-Schreier-Mumford curves we are studying, are unifor\-mized by $\Gamma=[A,B]\cong[\mathbb{Z}/p\mathbb{Z},  \mathbb{Z}/p\mathbb{Z}]$, and are 
given by the  following algebraic model
\[
X_c: (x^{p}-x)(y^{p}-y)=c,\]
for some $c\in K$, $|c|<1$. 
The group $\Z/p\Z$ is a  subgroup of the automorphism group and  acts for instance on the curve $X_c$ by letting  the generator $\tau$ of $\Z/p\Z$ to act on the curve in terms of the map $(x,y)\mapsto(x,y+1)$. We call $Y$  the quotient curve $X_c/\langle \tau \rangle$. Note that $Y$ is isomorphic to $\mathbb{P}^1,$ and hence  the genus $g_Y$ is zero.

Denote the function field of $ X_c$, for a fixed value of $|c|<1$, by $F$. The extension $F/K(x)$ 
is a cyclic extension of the  rational function field
$K(x)$. In this extension  $p$-places $P_{i}=(x-i)$, $1\leq i \leq p-1$ of $K(x)$ are weakly   ramified. 
The different is:
\[\mathrm{Diff}_{F/K(x)}=\sum_{i=0}^{p-1}2(p-1)P_{i}.\]
We will employ the  results of S. Nakajima  \cite{Nak:86}. We have the following decomposition 
in terms of indecomposable modules
\[
 H^0(X,\Omega_X^{\otimes n})= \bigoplus_{i=1}^p  m_i J_i,
\]
and the coefficients are given by \cite[Theorem 1]{Nak:86} (we have used the equivalent expression given at the final displayed equation of the proof of this theorem):
\[
 m_p=(2n-1)(g_Y-1) +\sum_{i=1}^p \lf \frac{n_i-(p-1)N_i}{p} \rf,
\]
where $N_i=1$ (ordinary curves) and $n_i=2n(p-1)$, see \cite[Section 4]{KoJPAA06}. 

Since $g_Y=0$ we compute:
\begin{eqnarray*}
 m_p &= & (2n-1)(g_Y-1)+ p \lf \frac{2n(p-1)-(p-1)}{p} \rf \\
     &= & (2n-1)(g_Y-1) + p (2n-1) - p \lc \frac{2n-1}{p} \rc \\
 &\stackrel{g_Y=0}{=}& (p-1) (2n-1) - p \lc \frac{2n-1}{p} \rc.
\end{eqnarray*}
For $1\leq j \leq p-1$ the coefficients $m_j$ are given by the following formulas \cite[Theorem 1]{Nak:86}:
\begin{eqnarray*}
 \frac{m_j}{p} & = &-\lf \frac{n_i-j N_i}{p}\rf
+
\lf \frac{n_i-(j-1)N_i}{p}\rf\\
 &=& 
-\lf \frac{2n(p-1)-j}{p}\rf
+\lf \frac{2n(p-1)-(j-1)}{p}\rf \\
&=&
-\lf \frac{-2n-j}{p} \rf +
\lf \frac{-2n-(j-1)}{p}\rf\\
&=&
\lc \frac{2n+j}{p} \rc - \lc \frac{2n+j-1}{p} \rc.
\end{eqnarray*}
We now notice that for $0 \leq j \leq p-1$ the above expression is zero unless
$p \mid 2n+j-1$.

We write $2n-1=\lf \frac{2n-1}{p} \rf p + r$, and we see that $m_j=0$ unless
\[
 j=p-r=p-(2n-1)+\lf \frac{2n-1}{p} \rf p.
\]
Notice that if $p> 2n-1$ then $j=p-(2n-1)$. So we have that 
\begin{equation} \label{H1}
 H^1(\Gamma,P_n)=H^0(X,\Omega_X^{\otimes n})= K[A]^{ (p-1) (2n-1) - p \lc \frac{2n-1}{p} \rc } \bigoplus  J_{p-r}^p.
\end{equation}

%
%
%

\subsection{Using the theory  of B. K\"ock. Study of the $K[A\times B]$-module 
structure} \label{Koeck}

In this section we will employ the results of B. K\"ock on the projectivity of the cohomology groups of certain sheaves in the weakly ramified case. 
Consider a $p$-group $G$ and   the cover $\pi:X \rightarrow X/G=:Y$.
For every point $P$ of $X$ we consider the local uniformizer $t$ at $P$, the stabilizer 
$G(P)$ of $P$  and assign a sequence of ramification groups 
\[
G_i(P)=\{\sigma \in G(P): v_P(\sigma(t)-t) \geq i+1\}.
\]
Notice that $G_0(P)=G(P)$ for $p$-groups, see \cite[chap. IV]{SeL}.
Let $e_i(P)$ denote the order of $G_i(P)$. 
We use the notation $X_{\mathrm{ram}}$ for the set of 
ramification points.
We will say that the cover $X \rightarrow X/G$
is weakly ramified if all  $e_i(P)$ vanish for $i \geq 2$.  All Mumford curves $X$  are ordinary and in $X \rightarrow X/\A(X)$ only weak ramification is allowed \cite{CKK}. We denote by $\Omega_X$ the sheaf of differentials on $X$ and by $\Omega_X(D)$ the 
sheaf $\Omega_X \otimes_{\mathcal{O}_X} \mathcal{O}_X(D)$. For a divisor $D=\sum_{P\in X} n_P P$
we denote by $D_{\mathrm{red}}=\sum_{P\in X:n_P \neq 0} P$ the associated reduced divisor. We will also denote by 
\[
L(D)=H^0(X,\mathcal{O}_{X}(D))=\{D +(f) >0: f \in F_X \} \cup \{0\},
\]
where $F_X$ is the function field of the curve $X$.
The ramification divisor equals $R=\sum_{P \in X}  \sum_{i=0}^\infty \big(e_i(P)-1 \big)$. Finally, $\Sigma$ denotes the skyscraper sheaf defined by the short exact sequence:
\begin{equation} \label{AA6} 0 \rightarrow \Omega_X^{\otimes n} \rightarrow \Omega_X^{\otimes n}\big( (2n-1) R_{\mathrm{red}} \big)
 \rightarrow \Sigma \rightarrow 0.
\end{equation}

\begin{lemma} For $n>1$
the cohomology group $H^1(X,\Omega_X^{\otimes n})=0$. 
\end{lemma}

\begin{proof}
There is a correspondence of sheaves between divisors and 1-dimensional $\mathcal O_X$-modules, $D\mapsto\mathcal O_X(D)$. 
The divisor of any differential is a  canonical divisor $K$ and $\Omega_X$ can be identified with  $\mathcal O(K)$. 

Recall that Serre duality asserts:
\[\dim H^1(X,\mathcal O_X(D))=\dim H^0(X,\Omega_X \otimes \mathcal O_X(D)^{-1} ).\]
Hence we find that
\[\dim H^1(X,\Omega_X^{\otimes n})= \dim H^0 (X, \Omega_X \otimes \Omega_X^{ - \otimes n}).\]
The sheaf $\Omega_X \otimes \Omega_X^{ - n}$ corresponds to the 
$\mathcal{O}_X$-module 
$\mathcal O_X(K-nK)$ and since  \[H^0(X,\mathcal O_X(K-nK)) = L( K-nK),\] it holds that 
\[\dim H^1(X,\Omega_X^{\otimes n})=\dim L(K-nK)=0.\]
\end{proof}

%

Now we apply the functor of global sections to the short exact sequence in (\ref{AA6}) and obtain the following short 
exact sequence:
\begin{equation} \label{BK-theory}
 0 \rightarrow H^0(X,\Omega_X^{\otimes n}) \rightarrow  
 H^0\big(X,\Omega_X^{\otimes n}\big( (2n-1) R_\mathrm{red}  \big) \big) \rightarrow 
 H^0(X,\Sigma) \rightarrow H^1(X,\Omega_X^{\otimes n})=0.
\end{equation}

\begin{theorem} \label{BK-dim} 
The $K[G]$-module $H^0\big(X,\Omega_X^{\otimes n}\big( (2n-1) R_\mathrm{red}  \big) \big)$  is a free 
$K[G]$-module of rank $(2n-1)(g_Y-1+r_0)$, where $r_0$ denotes the cardinality of 
$X^G_\mathrm{ram}=\{P \in X/G: e(P) > 1\}$, and $g_Y$ denotes the genus of the 
quotient curve $Y=X/G$.
\end{theorem}

\begin{proof}
 Since $G$ is a $p$-group a module is free if and only if it is projective. So we have to show that 
 $H^0\big(X,\Omega_X^{\otimes n}\big( (2n-1) R_\mathrm{red}  \big) \big)$ 
 is projective.
 B. K\"ock proved \cite[Theorem 2.1]{Koeck:04} that if $D=\sum_{P \in X} n_p P$ is a $G$-equivariant divisor,  
 the map $\pi: X \rightarrow Y:=X/G$ is weakly ramified, $n_P\equiv -1 \mod\; e_P$ for all $P \in X$ and 
 $\deg(D) \geq 2g_X-2$, then the module $H^0(X, \mathcal{O}_X(D))$ is projective. 
 
 We have to check the conditions for the divisor $D=nK_X +(2n-1)R_{\mathrm{red}}$, where $K_X$ is a 
 canonical divisor on $X$. Notice that $K_X=\pi^*K_Y+R$ and 
 $R=\sum_{P \in X} 2\big(e_0(P)-1 \big)$, therefore 
 \[
  D= n \pi^* K_Y + \sum_{P\in X: e_0(P)>1}\!\!\!\!\!\!\!\! \left( 2n e_0(P) -2n +2n-1  \right)P.
 \]
Therefore, the condition $n_P \equiv -1 \mod\; e_0(P)$ is satisfied.

We will now compute the dimension of $H^0\big(X,\Omega_X^{\otimes n}\big( (2n-1) R_\mathrm{red}  \big) \big)$ using 
Riemann--Roch theorem, keep in mind that $H^1\big(X,\Omega_X^{\otimes n}\big( (2n-1) R_\mathrm{red}  \big) \big)=0$
\begin{eqnarray*}
 \dim_K H^0\big(X,\Omega_X^{\otimes n}\big( (2n-1) R_\mathrm{red}  \big) \big) &= &
 n(2g_X-2) + (2n-1) |X_{\mathrm{ram}}| +1 - g_X \\
 &=&  (2n-1)(g_X-1+ |X_{\mathrm{ram}}|) \\
 &=& |G| (2n-1)(g_Y-1+ r_0),
\end{eqnarray*}
where in the last equality we have used the 
Riemann--Hurwitz formula  \cite[7, Corollary IV 2.4]{Hartshorne:77}
\[
 g_X-1=|G| (g_Y-1) + \sum_{P\in X_\mathrm{ram}} (e_0(P)-1). 
\]
\end{proof}

\begin{remark}
 This method was applied by the second author and B. K\"ock  in  \cite{kockKo} for the 
$n=2$ case in order to compute the dimension of the tangent space to the deformation functor of curves with automorphisms.
Deformations of curves with automorphisms for Mumford curves were also studied by the first author and 
G. Cornelissen in  \cite{CK}.
\end{remark}
 
The sequence in eq. (\ref{BK-theory}) leads to the  following short exact sequence of modules:
\begin{equation} \label{sesfin}
 0 \rightarrow H^0(X,\Omega_X^{\otimes n}) \rightarrow  
 K[G]^{ (2n-1) (g_Y-1+r_0)} \rightarrow 
 H^0(X,\Sigma) \rightarrow 0.
\end{equation}
Since $\Sigma$ is a skyscraper sheaf the space $H^0(X,\Sigma)$ is the direct sum of the stalks of $\Sigma$
\[
 H^0(X,\Sigma)= \bigoplus_{P \in X_\mathrm{ram}} \Sigma_P \cong \bigoplus_{j=1}^{r_0} \mathrm{Ind}_{G(P_j)}^G (\Sigma_{P_j}),
\]
where, for a subgroup $H$ of $G$, $\mathrm{Ind}_H^GV$ denotes the induced representation 
of an $K[H]$-module $V$ to a $K[G]$-module, i.e., $\mathrm{Ind}_H^G V=V \otimes_{K[H]}K[G]$.

\subsection{Return to Artin-Schreier-Mumford curves: proof of Theorem \ref{mainFree} (2)}
Recall that we are in the case $N=A*B$ and $\Gamma=[A,B]$, where $A\cong B \cong \Z/p\Z$. Set $G=N/\Gamma=\Z/p\Z\times \Z/p\Z$. 
\begin{lemma}
The indecomposable summands of  the module $\mathrm{Ind}_{G(P_j)}^G (\Sigma_{P_j})$ 
are either $K[G]$ or $K[G]/\langle (\sigma -1)^{\lambda} \rangle$, 
where $\sigma=\epsilon_A$ or $\epsilon_B$ and $1\leq \lambda\leq p-1$. 
\end{lemma}
\begin{proof}
It follows from the ramification of the function field of  Artin-Schreier-Mum\-ford  curves, seen as a double Artin-Schreier extension of the rational function field, where $r_0=2$, i.e., only two points $p_1,p_2$ of $X/{(A\times B)}$ are ramified in the cover $X \rightarrow X/(A\times B)$. Another way of obtaining this result is by using the theory of graphs of Mumford curves developed by the first author, and by noticing that the subgroup of the normalizer of the Artin-Schreier-Mumford  curve giving rise to the $A\times B$ cover is 
just $A*B$ corresponding to a graph with two ends, see \cite{CKNotices}, \cite[Proposition 5.6.2]{OrbifoldKato}.
Select a point $P_1$ of $X$ which lies above $p_1$ and a point 
$P_2$ of $X$ which lies above $p_2$. 
Let $G(P_1)=A$ and 
$G(P_2)=B$.

We will use an approach similar to \cite{kockKo} in order to study the Galois module 
structure of the stalk $\Sigma_{P_j}$ as a  $K[G_{P_j}]$-module. Let $P=P_j$ for  $j=1$ or $j=2$. 
Notice first that $nK_X=n\pi^* K_Y +nR$, so if the multiplicity   of the divisor  $K_Y$ at $\pi(P)$ is $m$, then 
the multiplicity of $nK_X$ at $P$ is $mnp+ 2n(p-1)$ and the 
multiplicity of $nK_X+(2n-1)R_{\mathrm{red}}$ at $P$
is $mnp+2n(p-1)+(2n-1)$. So if $t$ is a local uniformizer at $P$ and $s$ is a local uniformizer at 
$\pi(P)$ we have that:
\[
\Sigma_P= 
\left\langle 
\frac{t}{s^{mn+2n}}, \frac{t^2}{s^{mn+2n}},\ldots,\frac{t^{2n-1}}{s^{mn+2n}}
\right\rangle_K,
\]
which is 
 $G(P)$-equivariant isomorphic to the $K$-vector space generated  by: 
\[
\Sigma_P=\left\langle 
t,t^2,\ldots,t^{2n-1}
\right\rangle_K.
\] 
The action of $G(P)$ on $\Sigma_p$ is given by the transformation $\sigma(1/t) = 1/t+1$ for a generator 
$\sigma$ of the cyclic group $G(P)$, or equivalently $\sigma(t)=\frac{t}{1+t}$, see \cite{CK-Crelle}.
Notice,  that the element $t^{-p}-t^{-1}=\frac{1-t^{p-1}}{t^p}$ is invariant
and so is its inverse  $t^p (1-t^{p-1})^{-1}$. Here the unit $(1-t^{p-1})^{-1}$, can be seen as a polynomial 
modulo $t^{2n}$, if we expand it in terms of a geometric series and truncate 
the terms of degree $\geq 2n$. 
Now  we analyse the $G(P)$-module structure of $\Sigma_P$ using Jordan blocks.
Observe that for $0 \leq k \leq p-1$:
\[
\sigma \binom{1/t}{k}=\binom{1/t}{k} + \binom{1/t}{k-1},
\]
where 
\[
\binom{1/t}{k}=
\frac{1}{k!} \prod_{\nu=0}^{k-1} \left(\frac{1}{t}-\nu \right)=\frac{1}{k! t^k} \prod_{\nu=0}^{k-1} (1-t\nu). 
\]
Note that $\binom{1/t}{k}$ is a rational function, where the denominator 
is $k!t^k$. So if we multiply it by the invariant element
$t^p (1-t^{p-1})^{-1}$ we obtain a polynomial of degree $p-k$. Another $K$-vector space basis of $\Sigma_P$ is given by: 
\[
\left(
\frac{t^p}{(1-t^{p-1})} \right)^i \binom{1/t}{k}, 
\begin{array}{l} {\mbox{ where }
1\leq i \leq \lf \frac{2n-1}{p} \rf  \mbox{ and } 0 \leq k \leq p-1}
\\
{ \mbox{ or } i=\lf \frac{2n-1}{p} \rf+1 \mbox{ and } p-r \leq  k\leq p-1
}
\end{array}.
\]
The above defined elements are seen as polynomials by expanding them 
as powerseries and truncate the powers of $t$ greater than $2n$.
These polynomials, depending on $i$ and $k$, have degree $pi-k$.
Their degrees  start from degree one ($i=1,k=p-1$) to $2n-1$ ($i= \lf \frac{2n-1}{p} \rf +1, k=p-r$). 

For fixed $i$, $i=1,\dots,\lf \frac{2n-1}{p} \rf$,  and by allowing $k$ to vary from $0\leq k \leq p-1$, we obtain a Jordan block $J_p$.
The remaining block $i=\lf\frac{2n-1}{p} \rf+1$, $p-r\leq k \leq p-1$ is $J_r$.

So the structure of $\Sigma_P$ is given by 
\begin{equation} \label{SP}
\Sigma_P=  J_p^{ \lf \frac{2n-1}{p} \rf}  \bigoplus J_{r}.
\end{equation} 
 Recall \cite[12.16 p.74]{Curtis-Reiner} that if $H$ is a subgroup of $G$ and $g_1,\ldots,g_\ell$ is a set of 
 coset representatives of $G$ in $H$,  then for an $K[H]$-module $M$ the induced module can 
 be written as 
 \[
 \mathrm{Ind}_H^G M=\bigoplus_{\nu=1}^\ell g_\nu \otimes M.
 \] 
 Using the above equation for $G=A\times B$ and $H=G(P_1)=A$ (resp. $G(P_2)=B$) we 
 have
 \[ \mathrm{Ind}_{G(P_j)}^G(J_p)=K[G] \mbox{ and } 
 \mathrm{Ind}_{G(P_1)}^G(J_r)= 
 \frac{K[G]}{ (\epsilon_A -1)^{r} }.\]
Similarly 
\[
\mathrm{Ind}_{G(P_2)}^G(J_r)= 
 \frac{K[G]}{ (\epsilon_B -1)^{r} }\]
 and both of the above $K[G]$-modules are indecomposable. 
\end{proof}

\begin{proposition} 
 The indecomposable summands $V_i$ of $H^0(X,\Omega_X^{\otimes n})$ are either 
 $K[G]$ or $K[G]/\langle (\sigma -1)^{p-r} \rangle$, for $\sigma=\epsilon_A$ or $\sigma=\epsilon_B$ and  $r$ is the remainder of the division 
 $2n-1$ by $p$. 
\end{proposition}
\begin{proof}
Let $V_i$ be a indecomposable summand of $H^0(X,\Omega_X^{\otimes n})$. 
Consider the injective hull of $V_i$. This is the smallest injective module containing 
$V_i$, and it is of the form $K[G]^a$. Keep in mind that for group algebras of  finite 
groups the notions of injective and projective modules  coincide \cite[Theorem 62.3]{Curtis-Reiner}.

Therefore we have to consider the 
smallest 
$a$ such that $V_i \subset K[G]^a$.
We have  the short exact sequence:
\begin{equation} \label{Sha1}
 0 \rightarrow V_i \rightarrow K[G]^a \rightarrow {\Omega}^{-1}(V_i) \rightarrow 0, 
\end{equation}
where $\Omega^{-1}(M)$ for a a $K[G]$-module denotes the cokernel of the embedding of $M$ inside its injective hull. 
Since the algebra $K[G]$ is self injective (i.e., $K[G]$ is injective) we have
(for some appropriate natural number $t$)
\begin{equation} \label{vanishone}
V_i \cong \Omega (\Omega^{-1})(V_i) \bigoplus K[G]^t,
\end{equation}
where $\Omega(\Omega^{-1}(V_i))$ denotes the loop space of $\Omega^{-1}(V_i)$,  see \cite[Exercise 1 p.11]{Benson}. Since $V_i$ is indecomposable,
one of the two direct summands of eq. (\ref{vanishone}) is zero, so  either
$V_i \cong K[G]$ or $V_i=\Omega (\Omega^{-1})(V_i)$.
%
%

In the second case, we can consider 
the following diagram, where the first row comes from eq. (\ref{sesfin}) and 
the second by eq. (\ref{Sha1}):
\[
\xymatrix{
0 \ar[r] &  H^0(X,\Omega_X^{\otimes n})  \ar[r]  & K[G]^{(2n-1)(g_Y-1+r_0)} \ar[r]
& H^0(X,\Sigma) \ar[r] & 0 \\
0 \ar[r] & V_i \ar[r] \ar[u] & K[G]^a  \ar[r] \ar[u] & \Omega^{-1}(V_i) \ar[r] \ar[u] & 
0 
}.
\]
Notice that since $V_i$ is a direct summand of $H^0(X,\Omega_X^{\otimes n})$ which is contained in $K[G]^{(2n-1)(g_Y-1+r_0)}$ we can assume that the injective hull $K[G]^a$ of $V_i$ is a submodule of $K[G]^{(2n-1)(g_Y-1+r_0)}$.
 The module $\Omega^{-1}(V_i)$ is a non-zero  indecomposable non-projective factor of $H^0(X,\Sigma)$
and   is isomorphic to   $\mathrm{Ind}_{G(P_i)}^G(J_r)=K[G]/\langle (\sigma -1)^{r} \rangle$.
It can not be $K[G]$ since $K[G]$ is projective.
We compute  
\[V_i=\Omega(\mathrm{Ind}_{G(P_i)}^G(J_r))=\Omega(K[G]/\langle (\sigma -1)^{r}) \rangle)\cong K[G]/\langle (\sigma-1)^{p-r} \rangle.\]\end{proof}

\begin{corollary} \label{inv-deco}
 The space  $H^0(X,\Omega_X^{\otimes n})^G$ has dimension equal to the number of indecomposable summands.
\end{corollary}
\begin{proof}
Notice that each indecomposable summand $V_i$ is contained in a $K[G]$. 
\end{proof}

\begin{corollary}
 If $2n-1 \equiv 0 \mod \;
p$ then $H^0(X,\Omega^{\otimes n}) $ is projective. 
\end{corollary}

Now we finish the proof of \ref{mainFree} (2).
Using the sequence given in eq. (\ref{sesfin}) and the fact that only two 
points of $Y$  are ramified in $X\rightarrow Y$, i.e., $g_Y=0,r_0=2$, together with eq. (\ref{SP}) we obtain that the number of 
 summands which are isomorphic to $K[G]$ in $H^0(X,\Omega^{\otimes n})$ is $2n-1 - 2\lc \frac{2n-1}{p} \rc$.
There are two indecomposible  summands in $H^0(X,\Omega_X^{\otimes n})$, $V_1,V_2$ 
such that 
\[
K[G]/V_1= K[G]/h^r \mbox{ and } K[G]/V_2= K[G]/(\epsilon_B-1)^r.
\]
We see that 
\[
V_1=K[G]/h^{p-r} \mbox{ and } V_2=K[G]/(\epsilon_B-1)^{p-r}.
\]
Adding all these together we obtain:
\[
 H^0(X,\Omega_X^{\otimes n}) = K[G]^{2n-1 - 2\lc \frac{2n-1}{p} \rc} \bigoplus  K[G]/h^{p-r}
\bigoplus K[G]/(\epsilon_B-1)^{p-r}.
\]
The Proof of Theorem \ref{mainFree} (2) is now complete. 

 \def\cprime{$'$}
\providecommand{\bysame}{\leavevmode\hbox to3em{\hrulefill}\thinspace}
\providecommand{\MR}{\relax\ifhmode\unskip\space\fi MR }
\providecommand{\MRhref}[2]{%
  \href{http://www.ams.org/mathscinet-getitem?mr=#1}{#2}
}
\providecommand{\href}[2]{#2}

\end{document}